\documentclass[12pt]{amsart}

\usepackage{amssymb,latexsym,bbold,graphicx}

\usepackage{enumerate}

\makeatletter

\@namedef{subjclassname@2010}{ \textup{2010} Mathematics Subject Classification}

\makeatother
\newtheorem{thm}{Theorem}[section]
\newtheorem{Con}[thm]{Conjecture}

\newtheorem{cor}[thm]{Corollary}
\newtheorem{lem}[thm]{Lemma}
\newtheorem{pro}[thm]{Proposition}
\theoremstyle{definition}

\numberwithin{equation}{section}

\newcommand{\re}{\textup{Re}}

\newcommand{\h}{\mathcal{H}_{2g+1}}
\newcommand{\he}{\mathcal{H}_{2g+2}}
\newcommand{\hn}{\mathcal{H}_n}

\newcommand{\F}{\mathbb{F}_q[x]}

\newcommand{\legendre}[2]{\left(\frac{#1}{#2}\right)}
\newcommand{\s}{\sigma}

\begin{document}

\baselineskip=17pt

\title[values of $L$-functions in the hyperelliptic ensemble]{Complex Moments and the distribution of Values of $L(1,\chi_D)$ over Function Fields with Applications to Class Numbers}


\author{Allysa Lumley}

\address{Department of Mathematics and Statistics,
York University,
4700 Keele Street,
Toronto, ON,
M3J1P3
Canada}

\email{alumley@yorku.ca}

\date{}

\begin{abstract}  
In this paper we investigate the moments and the distribution of $L(1,\chi_D)$, where $\chi_D$ varies over quadratic characters associated to square-free polynomials $D$ of degree $n$ over $\mathbb{F}_q$, as $n\to\infty$. Our first result gives asymptotic formulas for the complex moments of $L(1,\chi_D)$ in a large uniform range. Previously, only the first moment has been computed due to work of Andrade and Jung. Using our asymptotic formulas together with the saddle-point method, we show that the distribution function of $L(1,\chi_D)$ is very close to that of a corresponding probabilistic model. In particular, we uncover an interesting feature in the distribution of large (and small) values of $L(1, \chi_D)$, that is not present in the number field setting. We also obtain $\Omega$-results for the extreme values of $L(1,\chi_D)$, which we conjecture to be best possible. 
Specializing $n=2g+1$ and making use of one case of Artin's class number formula, we obtain similar results for the class number $h_D$ associated to $\mathbb{F}_q(T)[\sqrt{D}]$. Similarly, specializing to $n=2g+2$ we can appeal to the second case of Artin's class number formula and deduce analogous results for $h_DR_D$ where $R_D$ is the regulator of $\mathbb{F}_q(T)[\sqrt{D}]$.


\end{abstract}



\maketitle

\section{Introduction}
An interesting and important problem in number theory is to understand the size of the class group, known as the class number, for a given field. The case of quadratic extensions of $\mathbb{Q}$ has a rich history of investigation which extends back to Gau\ss. Let $d$ be a fundamental discriminant and $h_d$ represent the class number of the field $\mathbb{Q}(\sqrt{d})$. Describing the extreme values of $h_d$ and the distribution of these values has been widely investigated. The main line of attack in this problem is to study the moments of $L(1,\chi_d)$, with $\chi_d$ taken as the Kronecker symbol $\left(\frac{d}{\cdot}\right)$. This approach works because of Dirichlet's class number formula.
Some recent notable papers discussing this problem are those of Granville and Soundararajan \cite{GranSound} and Dahl and Lamzouri \cite{AY}. The approach in these articles is to compare the complex moments of $L(1,\chi_d)$ to that of a random model and use the class number formula to apply this information to $h_d$.  

Here we discuss the adaptation of these techniques to study the class number, denoted as $h_D$, over function fields, $\mathbb{F}_q(T)$ with $q\equiv 1 (\bmod\, 4)$ and $D$ a monic square free polynomial in $\mathbb{F}_q[T]$. 
In this context, $h_D=|\text{Pic}(\mathcal{O}_D)|$, where $\text{Pic}(\mathcal{O}_D)$ is the Picard group of the ring of integers $\mathcal{O}_D\subseteq \mathbb{F}_q(T)(\sqrt{D(T)})$. Since $D$ is a square free polynomial we have that $\text{Pic}(\mathcal{O}_D)=\mathcal{C}l(\mathcal{O}_D)$, the class group of $\mathcal{O}_D$, which provides the justification for the name `class number'. 

  In 1992, Hoffstein and Rosen \cite{HoffRosen} investigated this question and obtained an average result by fixing the degree of the polynomial. The result is stated as: let $M$ be odd and positive then 
\begin{equation}\label{RoHoffbnd}
\frac1{q^M}\sum_{\substack{D \text{ monic}\\ \deg(D)=M}}h_D=\frac{\zeta_{\mathbb{F}_q[T]}(2)}{\zeta_{\mathbb{F}_q[T]}(3)}q^{(M-1)/2}-q^{-1},\end{equation}
where 
$$\zeta_{\mathbb{F}_q[T]}(s)=\sum_{f \text{ monic}}\frac1{|f|^s} \text{ for } \Re(s)>1,$$
 is the Riemann zeta function over $\mathbb{F}_q[T]$. Here the norm of $f\in\mathbb{F}_q[T]\setminus\{0\}$ is $|f|=q^{\deg(f)}$.  This result is directly comparable to Gau\ss's conjecture (proven by Siegel \cite{Siegel}) for class numbers of imaginary quadratic number fields.
  Finally, letting $q\to\infty$ one obtains an asymptotic formula which can be compared to the 2012 work of Andrade \cite{Andrade} described below. 
 
There are two limits that can be considered when studying problems over function fields. The first fixes the degree of the polynomial and lets the number of elements in the base field go to infinity as was done by Hoffstein and Rosen. The second fixes the number of elements in the base field and allows the degree of the polynomials to go to infinity. The result of Andrade \cite{Andrade} considers the second perspective. His article describes the mean value of $h_D$ by averaging over $\h$ the set  of monic, square free polynomials with degree $2g+1$. Proving that 
\begin{equation}\label{AndhD}
\frac1{|\h|}\sum_{D\in\h}h_D\sim \zeta_{\mathbb{F}_q[T]}(2)\prod_{P \text{ irreducible}}\left(1-\frac1{(|P|+1)|P|^2}\right)q^{g} \text{ as } g\to\infty.\end{equation}
 We remark that \eqref{RoHoffbnd} and \eqref{AndhD} have the same order of magnitude in the main term as can be seen by taking $M=2g+1$. 
 
  Now, for any monic $D\in \mathbb{F}_q[T]$ we have Dirichlet characters modulo $D$ on $\mathbb{F}_q[T]$, defined in Section \ref{Prelim}. The natural follow up to this is to define a Dirichlet $L$-function associated to such a character: 
\[L(s,\chi)=\sum_{f \text{ monic}}\frac{\chi(f)}{|f|^s}, \text{ for } s\in\mathbb{C}.\]
  Artin \cite{Artin} proved a class number formula valid over function fields which links $h_D$ to $L(1,\chi_D)$ where $\chi_D(\cdot)$ is the Kronecker symbol $\left(\frac{D}{\cdot}\right)$: 
 \begin{equation}\label{Artinform}L(1,\chi_D)=\frac{\sqrt{q}}{\sqrt{|D|}}h_D=q^{-g}h_D, \text{ for } D\in\h.\end{equation}
  To prove \eqref{AndhD} Andrade makes use of an approximate functional equation for $L(1,\chi_D)$ to show
\begin{equation}\label{AndL1chiD}
\frac1{|\h|}\sum_{D\in\h}L(1,\chi_D)\sim \zeta_{\mathbb{F}_q[T]}(2)\prod_{P \text{ irreducible}}\left(1-\frac1{(|P|+1)|P|^2}\right) \text{ as } g\to\infty,\end{equation}
  and then applies \eqref{Artinform}.
The main drawback to using the approximate functional equation is that is difficult to use it to calculate large moments of $L(1,\chi_D)$. 


 In this article, we shall investigate the distribution of $L(1,\chi_D)$ for $D\in \hn$ as $n\to \infty$, where 
 \begin{equation}\label{defHn}
 \hn=\{D\in\mathbb{F}_q[T]: D \text{ is monic, square free, } \deg(D)=n\}.
 \end{equation}
  To do this we will need to compute large complex moments of the associated $L(1,\chi_D)$. We approach the computation of such moments via a random model, a technique that has been used successfully in the study of quadratic number fields. 
 
For the remainder of the article the following notation will be fixed. Let $\mathbb{A}=\mathbb{F}_q[T]$ taking $q\equiv 1(\bmod\, 4)$ for simplicity. Here $\log$ denotes base $q$ logarithm, $\ln$ is the natural logarithm and $\log_{j}$ ( respectively $\ln_j$) represent the $j$-fold iterated logarithm. 
Finally, let $P$ represent an irreducible (prime) polynomial. We define the generalized divisor function $d_z(f)$ on its prime powers as 
\begin{equation}\label{gendivfundef}
d_z(P^a)=\frac{\Gamma(z+a)}{\Gamma(z)a!},
\end{equation}
and extend it to all monic polynomials multiplicatively. Then, we can express the complex moments of $L(1,\chi_D)$ as follows. 

\begin{thm}\label{momLthm}
Let $n$ a positive integer, and $z\in \mathbb{C}$ be such that $|z|\le \frac{n}{260\log(n)\ln\log(n)}$.
Then 
\begin{equation*}
\frac1{|\hn|}\sum_{D\in\hn}L(1,\chi_D)^z=\sum_{\substack{f\text{ monic}}}\frac{d_z(f^2)}{|f|^2}\prod_{P|f}\left(1+\frac1{|P|}\right)^{-1}\left(1+O\left(\frac1{n^{11}}\right)\right).
\end{equation*}
\end{thm}
The strategy for proving this, and a following result about the distribution of values, is to compare the distribution of $L(1,\chi_D)$ to that of a probabilistic random model: Let $\{\mathbb{X}(P)\}$ denote a sequence of  independent random variables indexed by the irreducible (prime) elements $P\in \mathbb{A}$,  
and taking the values $0,\pm1$ as follows
\begin{equation}\label{defrandvar}
\mathbb{X}(P)=\begin{cases}
0 & \text{ with probability }\frac1{|P|+1}\\
\pm1 & \text{ with probability }\frac{|P|}{2(|P|+1)}.
\end{cases}
\end{equation}
Let $f=P_1^{e_1}P_2^{e_2}\cdots P_s^{e_s}$ be the prime power factorization of $f$, then we extend the definition of $\mathbb{X}$  multiplicatively as follows
\begin{equation}\label{extendeddefX}\mathbb{X}(f)=\mathbb{X}(P_1)^{e_1}\mathbb{X}(P_2)^{e_2}\cdots \mathbb{X}(P_s)^{e_s}.\end{equation}
In this article we compare the distribution of $L(1,\chi_D)$ with  
\begin{equation}\label{randprod}L(1,\mathbb{X}):=\sum_{f \text{ monic}}\frac{\mathbb{X}(f)}{|f|}=\prod_{P \text{ irreducible}}\left(1-\frac{\mathbb{X}(P)}{|P|}\right)^{-1},\end{equation} 
which converges almost surely.  Further properties of this model will be discussed in Section \ref{PropRandProd}. 

For $\tau>0$, define 
\[\Phi_{\mathbb{X}}:=\mathbb{P}(L(1,\mathbb{X})>e^{\gamma}\tau) \text{ and } \Psi_{\mathbb{X}}(\tau):=\mathbb{P}\left(L(1,\mathbb{X})<\frac{\zeta_{\mathbb{A}}(2)}{e^{\gamma}\tau}\right).\]
We prove that the distribution of $L(1,\chi_D)$ is well approximated by the distribution of $L(1,\mathbb{X})$ uniformly in a large range.
\begin{thm}\label{distfuncchi}
Let $n$ be large.  Uniformly in $1\le \tau \le \log n-2\log_{2}n-\log_{3}n$ we have 
\[\frac1{|\hn|}|\{D\in\hn : L(1,\chi_D)>e^{\gamma}\tau\}|=\Phi_{\mathbb{X}}(\tau)\left(1+O\left(\frac{e^{\tau}(\log n)^2\log_{2}n}{ n}\right)\right),\]
and 
\[\frac1{|\hn|}|\{D\in\hn : L(1,\chi_D)<\frac{\zeta_{\mathbb{A}}(2)}{e^{\gamma}\tau}\}|=\Psi_{\mathbb{X}}(\tau)\left(1+O\left(\frac{e^{\tau}(\log n)^2\log_{2}n}{ n}\right)\right).\]

\end{thm}
And below we describe the asymptotic behaviour of $\Phi_{\mathbb{X}}$ and $\Psi_{\mathbb{X}}$.
\begin{thm}\label{distfuncX}
For any large $\tau$ we have 
\begin{equation}\label{PhiXtau}
\Phi_{\mathbb{X}}(\tau)=\exp\left(-C_1(q^{\{\log\kappa(\tau)\}})\frac{q^{\tau-C_0(q^{\{\log\kappa(\tau)\}})}}{\tau}\left(1+O\left(\frac{\log\tau}{\tau}\right)\right)\right),
\end{equation}
where $\kappa(\tau)$ is defined by \eqref{kappadef}, $C_0(t)=G_2(t)$, $C_1(t)=G_2(t)-G_1(t)$ and $G_i(t)$ are defined in \eqref{defG1} and \eqref{defG2} respectively.   Furthermore we have $$-\frac1{\ln q}+\ln(\cosh(c))/c-\tanh(c)<-C_1(q^{\{\log\kappa(\tau)\}})<\ln(\cosh(q))/q-\tanh(q),$$ where $c=1.28377...$. 
The same results hold for $\psi_{\mathbb{X}}$. \\
Additionally, if we let $0<\lambda<e^{-\tau}$, then 
\begin{equation}\label{PhiXtaulambda}
\Phi_{\mathbb{X}}(e^{-\lambda}\tau)=\Phi_{\mathbb{X}}(\tau)(1+O(\lambda e^{\tau}))\text{ and } \Psi_{\mathbb{X}}(e^{-\lambda}\tau)=\Psi_{\mathbb{X}}(\tau)(1+O(\lambda e^{\tau})).
\end{equation}
\end{thm}
Our Theorem \ref{distfuncX} should be compared to those of \cite{GranSound} and \cite{AY}, both of which study the behaviour of $L(1,\chi_d)$ over quadratic number fields.  The asymptotic behaviour of $\Phi_{\mathbb{X}}(\tau)$ is strikingly similar in both of these papers. In \cite{GranSound} the authors are studying the distribution of $L(1,\chi_d)$ over all fundamental discriminants $d$, $|d|\le x$, comparing it to a corresponding probabilistic model $L(1,\mathbb{X})$. In \cite{AY} the authors are studying the distribution of $L(1,\chi_d)$ over fundamental discriminants of the form $d=4m^2+1$, $m\ge 1$ and $d$ is square free. The restriction in \cite{AY} is used in order to study the behaviour of class numbers associated to such $d$, again comparing to a corresponding probabilistic model. In both papers $\Phi_{\mathbb{X}}(\tau)=\text{Prob}(L(1,\mathbb{X})>e^{\gamma}\tau)$. Each obtains:
$$\Phi_{\mathbb{X}}(\tau)=\exp\left(-C_1\frac{e^{\tau-C_0}}{\tau}+O\left(\frac{e^{\tau}}{\tau^2}\right)\right),$$ 
where 
\[C_1:=1 \text{ and }C_0:=\int_0^1\frac{\tanh(t)}{t}dt+\int_{1}^{\infty}\frac{\tanh(t)-1}{t}dt=0.8187\ldots \]
Similar behaviour appears when studying the distribution of Euler-Kronecker constants of quadratic fields, see \cite[Theorem 1.2]{Lam15} for details. 
As can be seen from the statement of Theorem \ref{distfuncX} we observe some pathological behaviour special to function fields. We no longer achieve two constants reflected above as  $C_0$ and $C_1$. In our case the value of both $C_0(q^{\{\log\kappa(\tau)\}})$  and $C_1(q^{\{\log\kappa(\tau)\}})$ varies, although they remain bounded as the argument varies between $1$ and $q$. Below is a graph of $C_0(t)$ for $1\le t< q$ taking $q=5$, and $q=9$ the first moduli which satisfy the hypothesis $q\equiv 1(\bmod\, 4)$.
\[\includegraphics[width=0.45\textwidth, height=0.45\textheight, keepaspectratio]{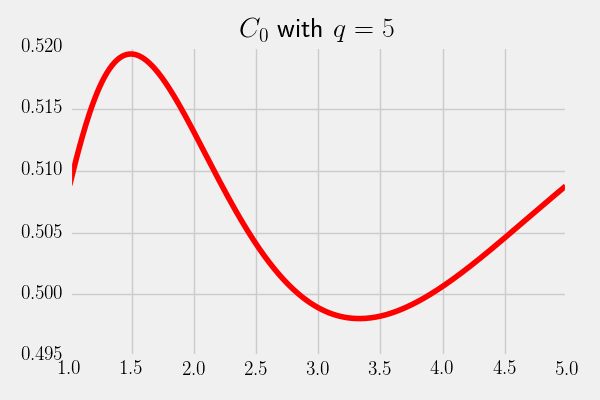}\qquad\includegraphics[width=0.45\textwidth, height=0.45\textheight, keepaspectratio]{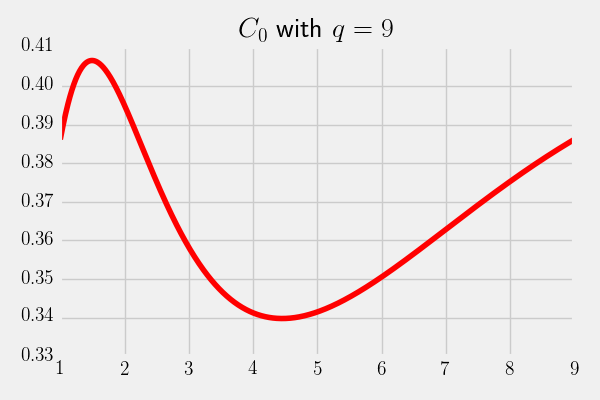}\]
Additionally, we also notice the coefficient $C_1$ which appears in all of the theorems describing the behaviour of $\Phi_{\mathbb{X}}(\tau)$( cf. \cite{GranSound,AY, Lam15}). We find over function fields that the coefficient $C_1$ is no longer fixed, but remains bounded between $-\ln(\cosh(q))/q+\tanh(q)$ and $1/\ln(q)-\ln(\cosh(c))/c+\tanh(c)$. Below is a graph of the behaviour of $C_1(t)$ for $1\le t<q$ with $q=5$ and $q=9$.
\[\includegraphics[width=0.45\textwidth, height=0.45\textheight, keepaspectratio]{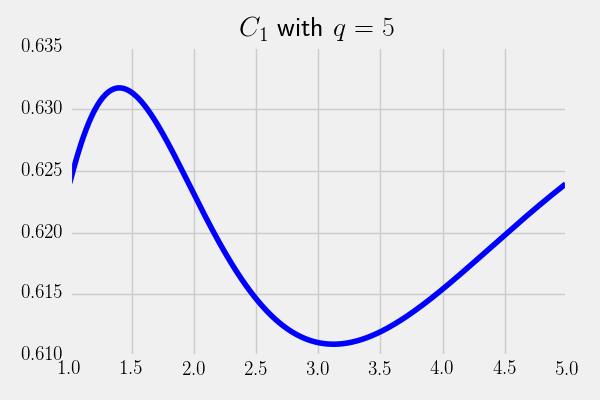}\qquad\includegraphics[width=0.45\textwidth, height=0.45\textheight, keepaspectratio]{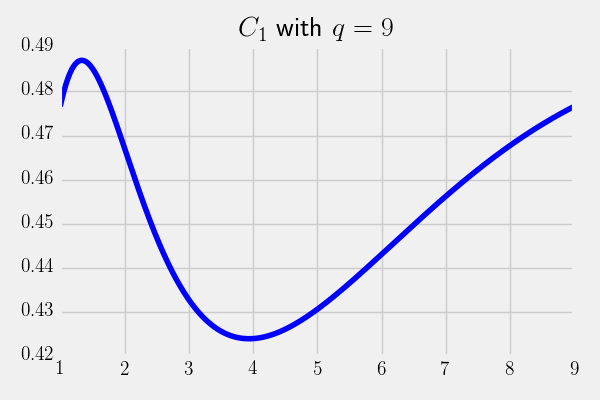}\]
Furthermore, we obtain the following unconditional bounds:
\begin{pro}\label{GRHBounds} 
Let $F$ be a monic polynomial, and $\chi$ be a non-trivial character on $\left(\mathbb{A}/F\mathbb{A} \right)^{\times}$. 
For any complex number $s$ with $\re(s)=1$ we have
\begin{equation}\label{BoundEdge}
\frac{\zeta_{\mathbb{A}}(2)}{2e^{\gamma}}(\log_2|F|+O(1))^{-1}\le|L(s, \chi)| \leq 2e^{\gamma}\log_2 |F|+ O(1).
\end{equation}
\end{pro}
It is important to note that in this setting Weil \cite{Weil} proved the Riemann Hypothesis (RH), hence these results are achieved unconditionally. We conjecture here that the true size for the extreme values of  $L(1,\chi_D)$ is half as large in keeping with the expected results in the quadratic number field case.
\begin{Con} \label{hDconjecture}
Let $n$ be large. 
\begin{equation*}
\max_{D\in\hn} L(1,\chi_{D})=e^{\gamma}(\log n+\log_2 n)+O(1),
\end{equation*}
and 
\begin{equation*}
\min_{D\in\hn}L(1,\chi_{D})= \zeta_{\mathbb{A}}(2)e^{-\gamma}(\log n+\log_2 n+O(1))^{-1}.
\end{equation*}
\end{Con}

Finally, we also unconditionally obtain $\Omega$-results which are best possible, unlike in the case of number fields where the corresponding bounds for Dirichlet characters is only valid under the Generalized Riemann Hypothesis (GRH).  
\begin{thm}\label{OmegaResults}
Let $N$ be large.
There are irreducible polynomials $Q_1$ and $Q_2$ of degree $N$ such that
\begin{equation}\label{Positive1}
L(1, \chi_{Q_1})\geq e^{\gamma} (\log_2|Q_1|+ \log_3 |Q_1|)+ O(1),
\end{equation}
and 
\begin{equation}\label{Negative1}
L(1, \chi_{Q_2})\leq \zeta_{\mathbb{A}}(2) e^{-\gamma} (\log_2 |Q_1|+ \log_3 |Q_1|+ O(1))^{-1}.
\end{equation}
\end{thm}
The result \eqref{Positive1} can be compared with 
\cite[Theorem 1]{AiMaPey} a recent work discussing the size of $|L(1,\chi)|$ over a number field. The authors prove using a variant of the resonator method that for $\epsilon>0$ and sufficiently large $d$ there is a character $\chi (\bmod\,\, d)$ such that 
\[|L(1,\chi)|\ge e^{\gamma}(\ln_2d+\ln_3d-(1+\ln_24)-\epsilon).\] 
This result provides an improvement over a paper of Granville and Soundararajan \cite{GranSound06}, however, the paper does not give improvements for quadratic characters $\chi_d$ where $d$ varies over fundamental discriminants in the range $|d|\le x$ cf. \cite{GranSound, Lam15IMRN}.   

The result \eqref{Negative1} can be compared to \cite[Theorem 5a]{GranSound} which under the assumption of GRH proves for any $\epsilon>0$ and all large $x$ there are $\gg x^{1/2}$ primes $d\le x$ such that 
\[L(1,\chi_d)\le \frac{\zeta(2)}{e^{\gamma}}(\ln_2 d+\ln_3 d-\ln_24-\epsilon)^{-1}.\]
Unconditionally, for $\chi$ a Dirichlet character modulo $d$ we have the weaker results $|L(1,\chi)|\le \frac{\zeta(2)}{e^{\gamma}}(\ln_2(d)-O(1))^{-1}$ from \cite{GranSound06} .
\subsection{Applications}\label{applications:intro} From the theorems above and in light of \eqref{Artinform} if we specialize $n$ as $n=2g+1$ and letting the genus $g\to\infty$ we can prove analogous results about the class number $h_D$ over $\h$. This specialization is the equivalent of studying the imaginary quadratic extensions of $\mathbb{Q}$, as described by Artin. Below we state a few of the resulting corollaries for $h_D$ with $D\in\h$.
\begin{cor}\label{momhDthm}
Let $z\in \mathbb{C}$ be such that $|z|\le \frac{g}{130\log(g)\ln\log(g)}$. 
Then 
\begin{equation*}
\frac1{|\h|}\sum_{D\in\h}h_D^z=q^{gz}\sum_{\substack{f\text{ monic}}}\frac{d_z(f^2)}{|f|^2}\prod_{P|f}\left(1+\frac1{|P|}\right)^{-1}\left(1+O\left(\frac1{g^{11}}\right)\right).
\end{equation*}
\end{cor}
This result follows from applying Artin's class number formula \eqref{Artinform} to Theorem \ref{momLthm} when $n=2g+1$. Additionally, from Theorems \ref{distfuncchi} and \ref{distfuncX} we obtain that the tail of the distribution of large (and small) values of $h_D$ over $\h$ is doubly exponentially decreasing:
\begin{cor}\label{disthD}
Let $g$ be large and $1\le \tau \le \log g-2\log_{2}g-\log_{3} g$. The number of discriminants $D\in \h$ such that 
\[h_D>e^{\gamma}\tau q^g\] 
equals
\[|\h|\cdot\exp\left(-C_1(q^{\{\log\kappa(\tau)\}})\frac{q^{\tau-C_0(q^{\{\log\kappa(\tau)\}})}}{\tau}\left(1+O\left(\frac{\log\tau}{\tau}\right)\right)\right),\] 
where $\kappa(\tau)$ is given by \eqref{kappadef},
$C_1(q^{\{\log\kappa(\tau)\}})$ and $C_0(q^{\{\log\kappa(\tau)\}})$ are positive constants depending on $\tau$ defined in Theorem \ref{distfuncX}. Similar estimates hold for the number of discriminants $D\in\h$ such that 
\[h_D< \frac{\zeta_{\mathbb{A}}(2)}{e^{\gamma}\tau}q^g.\]
\end{cor}
%

Similarly Proposition \ref{GRHBounds} give analogous upper and lower bounds and Theorem \ref{OmegaResults}   provides analogous Omega results for $h_D$ with $D\in\h$. 

Specializing to $n=2g+2$, we can also make connections to the class number $h_D$ for $D\in\he$. This case is analogous to studying a real quadratic extension of $\mathbb{Q}$ and so the class number formula changes. Indeed for for $D\in\he$ Artin proves:
\begin{equation}\label{realArtin}
L(1,\chi_D)=\frac{q-1}{\sqrt{|D|}}h_DR_D,
\end{equation}
where $R_D$ denotes the regulator of $\mathcal{O}_D$. In this case $R_D$ is defined to be $\log|\epsilon|_{P_{\infty}}$ where $\epsilon$ is a fundamental unit of $\mathcal{O}_D$, $P_{\infty}$ is the prime at infinity such that $\text{ord}_{P_{\infty}}(\epsilon)<0$ and $$\log|\epsilon|_{P_{\infty}}=-\deg(P_{\infty})\text{ord}_{P_{\infty}}(\epsilon).$$  For more details on the regulator see \cite[Chapter 14]{Ro3}. The case of the mean value for $L(1,\chi_D)$ taken over $\he$ was investigated by Jung \cite{Jung1, Jung2}. Taking $n=2g+2$ we deduce from \eqref{realArtin} and Theorem \ref{momLthm}:
\begin{cor}\label{momhDRD}
Let $z\in \mathbb{C}$ be such that $|z|\le \frac{g}{130\log(g)\ln\log(g)}$. 
Then 
\begin{equation*}
\frac1{|\he|}\sum_{D\in\he}(h_DR_D)^z=\left(\frac{q^{g+1}}{q-1}\right)^z\sum_{\substack{f\text{ monic}}}\frac{d_z(f^2)}{|f|^2}\prod_{P|f}\left(1+\frac1{|P|}\right)^{-1}\left(1+O\left(\frac1{g^{11}}\right)\right).
\end{equation*}
\end{cor}
Of course, similar results about the distribution of $h_DR_D$, upper and lower bounds and omega results for $D\in\he$ follow from Theorems \ref{distfuncchi} and \ref{distfuncX}, Proposition \ref{GRHBounds} and Theorem \ref{OmegaResults} respectively.

Finally, we give the outline of the paper. Section \ref{Prelim} will establish some facts about $\mathbb{A}$ and the properties $L$-functions have over this ring. Section \ref{complexmoments} will connect the complex moments of $L(1,\chi_D)$ to the expectation of the complex moments of the random model and provide the proof of Theorem \ref{momhDthm}. Section \ref{distrandmod} will be used to prove Theorem \ref{distfuncX}. Section \ref{MainThmProofs} proves Theorem \ref{distfuncchi} and Corollary \ref{disthD}. Section \ref{Optomegsection} proves the $\Omega$-results of Theorem \ref{OmegaResults}. 

\section{Preliminaries}\label{Prelim}
\subsection{Background for Function Fields}\label{functionfieldfacts} 

The norm of $f\in \mathbb{A}\setminus\{0\}$ is $|f|=q^{\deg(f)}$ and $|0|=0$. From \cite[Chapter 2]{Ro3} the Riemann zeta function is given by
\[\zeta_{\mathbb{A}}(s)=\prod_{\substack{P \text{ monic}\\\text{ irreducible}}}\left(1-\frac1{|P|^s}\right)^{-1}=\sum_{f \text{ monic}}\frac1{|f|^s}, \,\,\Re(s)>1.
\]
Since there are $q^k$ different monic polynomials of degree $k$, we can also rewrite $\zeta_{\mathbb{A}}(s)$ by collecting all the terms with respect to their degree: 
\[\zeta_{\mathbb{A}}(s)=\sum_{f \text{ monic}}\frac1{|f|^s}= \sum_{k\ge 0}\frac1{q^{k(s-1)}}=\frac1{1-q^{1-s}}. 
\]
Which is valid for $s\in\mathbb{C}\setminus\{1\}$. 
Let
\[\pi_q(n)=\#\{a \,\,|\,\, a \text{ monic}, \deg(a)=n  \text{ and } a \text{ is irreducible}\}.\]
The prime number theorem for polynomials gives the following about $\pi_q(n)$ (cf. \cite[Theorem 2.2]{Ro3}):
\begin{equation}\label{PNT1}
\sum_{k | m} k \pi_q(k)= q^m,
\end{equation}
and
\begin{equation}\label{PNT2}
\pi_q(n)=\frac{q^n}{n}+O\left(\frac{q^{n/2}}{n}\right).
\end{equation}
Let $F\in\mathbb{A}$ such that $\deg(F)>0$. From \cite[Chapter 4]{Ro3}, a { Dirichlet character modulo $F$}, $\chi: \mathbb{A}:\to\mathbb{C}$, satisfies
\begin{enumerate}
\item $\chi(a+bF)=\chi(a)$ for all $a,b\in\mathbb{A}$,
\item $\chi(a)\chi(b)=\chi(ab)$ for all $a,b\in\mathbb{A}$,
\item $\chi(a)\neq0 \Leftrightarrow (a,F)=1$.
\end{enumerate}
Then a Dirichlet $L$-function over $\mathbb{A}$ is given by 
\[L(s,\chi)=\sum_{f \text{ monic}}\frac{\chi(f)}{|f|^s}=\prod_{\substack{P \text{ monic}\\\text{ irreducible}}}\left(1-\frac{\chi(P)}{|P|^s}\right)^{-1} \text{ for } s\in\mathbb{C}.
\]
As with $\zeta_{\mathbb{A}}(s)$ we may collect terms with respect to the degree of the polynomial and write $L(s,\chi)$ as follows: 
\[L(s,\chi)=\sum_{f \text{ monic}}\frac{\chi(f)}{|f|^s}
=\sum_{k\ge 0}\frac1{q^{ks}}\sum_{\substack{f \text{ monic}\\ \deg(f)=k}}\chi(f).
\]

By \cite[Proposition 4.3]{Ro3}, if $\chi$ is a nontrivial Dirichlet character and $k\geq \deg F$ then 
\begin{equation}\label{CharSum}
\sum_{\substack{f \text{ monic}\\ \deg(f)=k}} \chi(f)= 0.
\end{equation}
That is to say that $L(s,\chi)$ is actually a polynomial in $q^{-s}$, whose degree is at most $\deg(F)-1$. Hence, we may also express it as a finite product of linear terms :
  \noindent $(1-\alpha_j(\chi)q^{-s})$, for $j=1,2,\ldots, n\le\deg(F)-1$.

Let $\Lambda(f)=\deg P$ if $f=P^k$ and $0$ otherwise, the function field analogue of the Von Mangoldt function, then from the proof of \cite[Theorem 4.8]{Ro3} we see 
\begin{equation}\label{CharSumPrimes}
\sum_{\substack{f \text{ monic}\\ \deg(f)=k}} \Lambda(f) \chi(f)=-\sum_{j=1}^{\deg F-1} \alpha_j(\chi)^k.
\end{equation}
 We mentioned in the introduction that A. Weil \cite{Weil} proved the analogue of the Riemann Hypothesis. In this setting this says that $|\alpha_j(\chi)|=1$ or $|\alpha_j(\chi)|=\sqrt{q}$. From this we deduce 
\begin{equation}\label{CharSumGRH}
\sum_{\deg P=k} \chi(P) \ll \frac{q^{k/2}}{k}\deg F,
\end{equation}
and that the Euler product representation of $L(s, \chi)$
$$ L(s, \chi)= \prod_{P} \left(1-\frac{\chi(P)}{|P|^s}\right)^{-1},$$ is actually valid for $\re(s)>1/2$. \\
One final remark about the size of $\hn$ defined in \eqref{defHn}, if $n>1$ then from \cite[Proposition 2.3]{Ro3}:
\[|\hn|=q^{n-1}(q-1).\]
Finally, let $D\in \h$. Consider 
\[C_D:y^2=D(x).\]
This defines a hyperelliptic curve over $\mathbb{F}_q$ with genus $g$. In this instance $h_D$ is associated to the number of $\mathbb{F}_q$-rational points on the Jacobian of $C_D$.
 
 Let $u=q^{-s}$ then the zeta function associated to  $C_D$  is defined by
\[Z_{C_D}(s)=\frac{P_{C_D}(u)}{(1-u)(1-qu)}.\]

Weil \cite{Weil} proved  $P_{C_D}(u)$ is a polynomial of degree $2g$. 
 In fact, from \cite[Propositions 14.6 and 17.7]{Ro3} we have 
\[P_{C_D}(u)=L(s,\chi_D)=\sum_{\text{sgn}(f)=1}\frac{\chi_D(f)}{|f|^s},\] 
where $\chi_D(f)$ is given by the Kronecker symbol: 
\[ \chi_D(f)=\left(\frac{D}{f}\right)_2=\left(\frac{D}{f}\right).\]
So from the point of view of $\eqref{Artinform}$ it's natural that $h_D$ should be associated to this hyperelliptic curve.
We note, if $q\equiv 1\pmod  4$ from quadratic reciprocity (cf. \cite[Theorem 3.5]{Ro3}) for any monic polynomials $F, G$ we have
\begin{equation}\label{QuadRecip}
\left(\frac{F}{G}\right)= \left(\frac{G}{F}\right),
\end{equation}
which explains the assumption we make on $q$. 
\subsection{Estimates for sums over irreducible monic polynomials}

Here and throughout we let $\Pi_q(n)$ be the number of monic irreducible polynomials $P$ such that \\
$\deg P\leq n$. 
\begin{lem}\label{PrimeSums}
Let $M$ be a large positive integer.
Then we have
\begin{equation}\label{NumberPrimes}
\Pi_q(M) = \zeta_{\mathbb{A}}(2)\frac{q^M}{M}\left(1+O\left(\frac{\log M}{M}\right)\right).
\end{equation}
\end{lem}
\begin{proof}
Note that
$$\Pi_q(M)=\sum_{n=1}^M \pi_q(n)=\sum_{n=1}^{M}\left(\frac{q^n}{n}+O\left(\frac{q^{n/2}}{n}\right)\right),$$
where the last equality comes from the Prime Number Theorem.  The main term in this sum is $q^M/M$, and we see that if $n\le M-\log M,$ then $q^n/n\ll q^M/M^2$. Hence we have that 
\[\Pi_q(M)=\sum_{M-\log M< n\le M}\frac{q^n}{n}+O\left(\frac{q^M}{M^2}\right).\]
Then for  $n\in(M-\log M, M]$ we have $\frac1n=\frac1M\left(1+O\left(\frac{\log M}{M}\right)\right).$ Therefore, 
\[\Pi_q(M)=\frac{q^M}M\left(1+O\left(\frac{\log M}{M}\right)\right)\sum_{l<\log M}\frac{1}{q^l}=\zeta_{\mathbb{A}}(2)\frac{q^M}M\left(1+O\left(\frac{\log M}{M}\right)\right).\]
 \end{proof}
\begin{lem}\label{Truncation}
Let $F$ be a monic polynomial, and $\chi$ be a non-trivial character on $\left(\mathbb{A}/\mathbb{A} F\right)^{\times}$. For a positive integer $M$ and any complex number $s$ with $\re(s)=1$
 we have
$$\ln L(s, \chi )= -\sum_{\deg P\le M}\ln\left(1-\frac{\chi(P)}{|P|^s}\right)+ O\left(\frac{q^{-M/2}}{M}\deg F\right). 
$$
\end{lem}
\begin{proof}
Split the sum as
\begin{align*}
\ln L(s,\chi)&=-\sum_{\deg P\le M}\ln\left(1-\frac{\chi(P)}{|P|^s}\right)-\sum_{\deg P>M}\ln\left(1-\frac{\chi(P)}{|P|^s}\right).\\
\end{align*}
The error term follows from \eqref{CharSumGRH} as below
\begin{align*}
\sum_{k>M}\sum_{\deg P=k}\ln\left(1-\frac{\chi(P)}{|P|}\right)&=\sum_{k>M}\sum_{\deg P=k}\frac{\chi(P)}{|P|^s}+O\left(\sum_{\deg P>M}\frac1{|P|^2}\right)\\
&=\sum_{k>M}\frac1{q^k}\sum_{\deg P=k}\chi(P)+O(q^{-M})\\
&\ll \deg F\sum_{k>M}\frac1{kq^{k/2}}\\
&\ll \deg F \frac{q^{-M/2}}{M}.
\end{align*}
\end{proof}
\noindent
%

We now prove a refined form of a Mertens' type estimate due to Rosen \cite{Ro1}.
\begin{lem}\label{Mertens}
Let $M$ be large. Then, we have 
$$ -\sum_{\deg P\leq M} \ln\left(1-\frac{1}{|P|}\right)= \ln M+\gamma+\frac{1}{2M}+O\left(\frac{1}{M^2}\right).
$$
\end{lem}
\begin{proof}
We have
$$ -\sum_{\deg P\leq M} \ln\left(1-\frac{1}{|P|}\right)= \sum_{\deg P\leq M} \sum_{\ell =1}^{\infty} \frac{1}{\ell |P|^{\ell}} =\sum_{k\leq M} \sum_{\ell =1}^{\infty} \frac{\pi_q(k)}{\ell q^{\ell k}}=\sum_{k\ell\leq M}\frac{\pi_q(k)}{\ell q^{\ell k}}+\sum_{\substack{k\leq M\\ k\ell>M}} \frac{\pi_q(k)}{\ell q^{\ell k}}.
$$
By making the change of variables $m=k\ell$ and using $\eqref{PNT1}$, we deduce that the first sum on the right hand side of the last identity equals
$$ \sum_{k\ell\leq M}\frac{\pi_q(k)}{\ell q^{\ell k}}=\sum_{m\leq M} \frac{1}{q^m m}\sum_{k | m} k \pi_q(k)=\sum_{m\leq M}\frac{1}{m}=\ln M+\gamma+\frac{1}{2M} +O\left(\frac{1}{M^2}\right).$$
The result follows upon noting that
$$ 
\sum_{\substack{k\leq M\\ k\ell>M}} \frac{\pi_q(k)}{\ell q^{\ell k}}\ll \sum_{\substack{k\leq M\\ k\ell>M}} q^{k(1-\ell)} \ll \sum_{2\leq \ell\leq M}\sum_{\frac{M}{\ell}<k\leq M} q^{k(1-\ell)}+ q^{-M}\ll q^{-M} \sum_{2\leq \ell\leq M} q^{\frac{M}{\ell}}\ll q^{-M/2}.
$$
\end{proof}
\subsection{Proof of Proposition \ref{GRHBounds}}
\begin{proof}[Proof of Proposition \ref{GRHBounds}]

For $\re(s)=1$,  we use Lemma \ref{Truncation} and together with the choice $M=2\log\log|F|$ to get
$$\ln L(s, \chi)=-\sum_{\deg P\leq M} \ln\left(1-\frac{\chi(P)}{|P|^{s}}\right)+ O\left(\frac{1}{M}\right).$$ 
Using this estimate together with Lemma \ref{Mertens} we deduce that
$$ |L(s, \chi)|\leq \prod_{\deg P\leq M}\left(1-\frac{1}{|P|}\right)^{-1}\left(1+O\left(\frac{1}{M}\right)\right) \leq e^{\gamma} M +O(1),
$$
which completes the proof of the upper bound in \eqref{BoundEdge}. To see the lower bound, note that from Lemma \ref{Truncation} we have 

\[|L(s,\chi)|\ge \prod _{\deg(P)\le M}\left(1+\frac{1}{|P|}\right)^{-1}\left(1+O\left(\frac{1}{M}\right)\right),\]
and from Lemma \ref{Mertens}
\[\prod _{\deg(P)\le M}\left(1+\frac{1}{|P|}\right)^{-1}=\prod _{\deg(P)\le M}\frac{\left(1-\frac{1}{|P|^2}\right)^{-1}}{\left(1-\frac{1}{|P|}\right)^{-1}}\ge \frac{\zeta_{\mathbb{A}}(2)}{e^{\gamma}M+O(1)}.\]
\end{proof}
\subsection{Sums over $\hn$.}
The orthogonality relation:

\begin{lem}\label{orthogonality}
Let $f$ be a monic polynomial. If $f$ is a square in $\mathbb{A}$, then
$$ \sum_{D\in \hn} \chi_D(f)= |\hn| \cdot \prod_{P \mid f} \left(1+\frac{1}{|P|}\right)^{-1}+ O\left(\sqrt{|\hn|}\right).$$
Furthermore, if $f$ is not a square in $\mathbb{A}$, then
$$ \sum_{D\in \hn} \chi_D(f) \ll \sqrt{|\hn|} \cdot 2^{\deg f}.$$
\end{lem}
\begin{proof}
The first estimate follows from Proposition 5.2 of \cite{AnKe}, while the second follows from Lemma 6.4 of \cite{AnKe}.
\end{proof}
\section{Complex moments of $L(1, \chi)$}\label{complexmoments}

Let $D\in\hn$, $z\in\mathbb{C}$ such that $|z|\ll \log|D|/(\log_2|D|\ln\log_2|D|)$. Let $\chi_D(f)=\legendre{D}{f}$.
 We we recall that $d_z(f)$ is defined as in \eqref{gendivfundef}. We will prove the following key lemma which will allow us to connect our complex moments of the random model to the complex moments of $L(1,\chi_D)$.

\begin{lem}\label{keylemma}
Let $D\in \hn$. Let $A>4$ be a constant $z\in\mathbb{C}$ such that $|z|\le \frac{\log|D|}{10A\log_2|D|\ln\log_2|D|}$ and $M=A\log_2|D|$.
Then 
\[L(1,\chi_D)^z=\left(1+O\left(\frac1{(\log|D|)^B}\right)\right)\sum_{\substack{f \text{ monic}\\|f|\le |D|^{1/3}\\P|f\Rightarrow \deg P\le M}}\frac{\chi_D(f)d_z(f)}{|f|},\]
where $B=A/2-2$.
\end{lem}
Before giving the proof, we make some estimates: 

 \begin{lem}\label{truncationpart2}
 Let $D\in \hn$, $A>4$ be a fixed constant and $z\in\mathbb{C}$ such that $|z|\le \frac{\log|D|}{10A\log_2|D|\ln\log_2|D|}$ and $M=A\log_2|D|$. Then for $c_0$ some positive constant we have 
 \begin{equation}\label{truncsumfterm}
\sum_{\substack{f \text{ monic}\\ P|f\Rightarrow \deg P\le M}}\frac{\chi_D(f)}{|f|}d_z(f)=\sum_{\substack{f \text{ monic}\\|f|\le|D|^{1/3}\\ P|f\Rightarrow \deg P\le M}}\frac{\chi_D(f)}{|f|}d_z(f)
 +O\left(|D|^{-\frac1{c_0\log_2|D|}}\right),
 \end{equation}
and furthermore, 
\begin{multline}\label{truncsumfprod}
 \sum_{\substack{f \text{ monic}\\ P|f\Rightarrow \deg P\le M}}\frac{\chi_D(f)}{|f|}d_z(f)\prod_{P|f}\left(1+\frac1{|P|}\right)^{-1}=\sum_{\substack{f \text{ monic}\\|f|\le|D|^{1/3}\\ P|f\Rightarrow \deg P\le M}}\frac{\chi_D(f)}{|f|}d_z(f)\prod_{P|f}\left(1+\frac1{|P|}\right)^{-1}
 \\+O\left(|D|^{-\frac1{c_0\log_q\log_q|D|}}\right).
\end{multline}
 \end{lem}
\begin{proof}
First we prove \eqref{truncsumfterm}. Let $z\in \mathbb{C}$  and let $k\in \mathbb{Z}$ such that $|z|<k$. Let $0<\alpha\le \frac12$, then using Rankin's trick we see 

\begin{align*}
\left|\sum_{\substack{f \text{ monic}\\ |f|>|D|^{1/3}\\P|f \Rightarrow \deg P\le M}}\frac{\chi_D(f)}{|f|}d_z(f)\right|&\le \sum_{\substack{f \text{ monic}\\ |f|>|D|^{1/3}\\P|f \Rightarrow \deg P\le M}}\frac{d_k(f)}{|f|} \le |D|^{-\alpha/3}\sum_{\substack{f \text{ monic}\\P|f \Rightarrow \deg P\le M}}\frac{d_k(f)}{|f|^{1-\alpha}}\\
&=|D|^{-\alpha/3}\exp\left(k\sum_{\deg P\le M}\frac1{|P|^{1-\alpha}}+O(k)\right).
\end{align*}
Choosing $\alpha=\frac1M$ we have that $|P|^{\alpha}=q^{\alpha\deg P}\le q=O(1)$, so that 
\begin{align*}
|D|^{-\alpha/3}\exp(k\sum_{\deg P\le M}\frac1{|P|^{1-\alpha}}+O(k)) &\ll |D|^{-\frac1{3M}}\exp\left(O\left(k\sum_{\deg P\le M}\frac1{|P|}\right)\right)\\
&\ll |D|^{-\frac1{3M}}\exp O\left(k \log M\right),
\end{align*}
by Mertens' theorem.
Taking $k \ll \frac{\log |D|}{\log_2|D|\ln\log_2|D|}$, and using $M=A\log_2|D|$ the expression inside of the big Oh becomes
\begin{equation*}
k\ln(A\log_2|D|)\ll\frac{\log|D|\ln( A\log_2|D|)}{\log_2|D|\ln\log_2|D|}.
\end{equation*}
So we have 
\begin{align*}
|D|^{-\frac1{3M}}\exp O\left(k\log M\right)
&\ll|D|^{-\frac1{c_0\log_2|D|}},
\end{align*}
for some $c_0>0$. 

The proof of \eqref{truncsumfprod} follows from the previous argument since \\
$$\prod_{P|f}\left(1+\frac1{|P|}\right)^{-1}\le 1.$$
\end{proof}

\begin{proof}[Proof of Lemma \ref{keylemma}]
From Lemma \ref{Truncation} we have 
\begin{align*}
 L(1,\chi_D)^z&=\exp(z\ln(L(1,\chi_D))\\
 &=\exp\left(-z\sum_{\deg P\le M}\ln\left(1-\frac{\chi_D(P)}{|P|}\right)\right)\exp\left(O\left(\frac{q^{-M/2}}{M}\deg D|z|\right)\right).
\end{align*}
Here we use the fact that $M=A\log_2|D|$ implying that $q^{-M/2}=(\log|D|)^{-A/2}$, $\deg D=\log|D|$, $|z|\le \frac{\log|D|}{10A\log_2|D|\ln\log_2|D|}$ to see that the expression inside of the big Oh has the shape 
\[\frac{(\log|D|)^2}{(\log|D|)^{A/2}}\frac1{10A^2(\log_2|D|)^2\ln\log_2|D|}=O\left(\frac1{(\log|D|)^{B}}\right),\]
by the assumption on $A$. Hence, we have
\begin{align}
\nonumber L(1,\chi_D)^z &=\prod_{\deg P\le M}\left(1-\frac{\chi_D(P)}{|P|}\right)^{-z}\left(1+O_{\s}\left(\frac1{(\log|D|)^{B}}\right)\right)\\
\nonumber&=\left(1+O\left(\frac1{(\log|D|)^{B}}\right)\right)\prod_{\deg P\le M}\left(\sum_{a=0}^{\infty}\frac{\chi_D(P^a)}{|P|^a}d_z(P^a)\right)\\
\nonumber&= \left(1+O\left(\frac1{(\log|D|)^{B}}\right)\right)\sum_{\substack{f \text{ monic}\\ P|f \Rightarrow \deg P\le M}}\frac{\chi_D(f)}{|f|}d_z(f).
\end{align}
Finally we apply \eqref{truncsumfterm} from Lemma \ref{truncationpart2}, then the relative sizes of the O terms completes the result for this case.
\end{proof}

Using this Lemma we have that  
\begin{align*}
\sum_{D\in\hn}L(1,\chi_D)^z&=\sum_{D\in\hn}\left(1+O\left(\frac1{(\log|D|)^B}\right)\right)\sum_{\substack{f \text{ monic}\\ |f|\le |D|^{1/3}\\ P|f\Rightarrow \deg P\le M}}\frac{\chi_D(f)}{|f|}d_z(f)\\
&=\left(1+O\left(\frac1{(\log|D|)^B}\right)\right)\sum_{\substack{f \text{ monic}\\ |f|\le |D|^{1/3}\\ P|f\Rightarrow \deg P\le M}}\frac{d_z(f)}{|f|}\sum_{D\in\h}\chi_D(f)\\
&=\left(1+O\left(\frac1{(\log|D|)^B}\right)\right)(S_1+S_2),
\end{align*}
where 
\begin{equation}\label{defS1}
S_1:=\sum_{\substack{f \text{ monic and a square}\\ |f|\le |D|^{1/3}\\ P|f\Rightarrow \deg P\le M}}\frac{d_z(f)}{|f|}\sum_{D\in\hn}\chi_D(f),
\end{equation}
and 
\begin{equation}\label{defS2}
S_2:=\sum_{\substack{f \text{ monic and not a square}\\ |f|\le |D|^{1/3}\\ P|f\Rightarrow \deg P\le M}}\frac{d_z(f)}{|f|}\sum_{D\in\hn}\chi_D(f).
\end{equation}
With this separation we can use our orthogonality relation to evaluate $S_1$ and $S_2$. 
\subsection{Evaluating $S_2$: Contribution of the non-square terms.}
\begin{lem}
Let $D\in \hn$, $A>4$ be a constant, $z\in \mathbb{C}$ be such that $|z|\le \frac{\log|D|}{\log_2|D|\ln\log_2|D|}$ and $M=A\log_2|D|$. Then 
\[\left(1+\frac1{(\log|D|)^{B}}\right)S_2\ll |D|^{6/7},\]
with $S_2$ defined as in \eqref{defS2}.
\end{lem}
\begin{proof}
By Lemma \ref{orthogonality} the inner sum of $S_2$ is  $O(\sqrt{|\hn|}2^{\deg f})$, hence we have 
\begin{align*}
S_2\ll \sqrt{|\hn|}\sum_{\substack{f \text{ monic and not a square}\\ |f|\le |D|^{1/3}\\ P|f\Rightarrow \deg P\le M}}\frac{d_z(f)2^{\deg f}}{|f|}.
\end{align*}
 Now, we have that $|\hn|=O(|D|)$ and $2^{\deg f}=|f|^{\ln2/\ln q}<|f|^{11/25}$, the last inequality following from the fact that $q\ge 5$ and $\ln 2/\ln 5=0.43067\ldots$. Thus 
\begin{align*}
S_2&\ll |D|^{1/2}\sum_{\substack{f \text{ monic and not a square}\\ |f|\le |D|^{1/3}\\ P|f\Rightarrow \deg P\le M}}\frac{d_z(f)}{|f|^{14/25}},\\
&\ll |D|^{5/6}\sum_{\substack{f \text{ monic and not a square}\\ |f|\le |D|^{1/3}\\ P|f\Rightarrow \deg P\le M}}\frac{d_z(f)}{|f|^{39/25}},\\
&\ll |D|^{5/6}(\zeta_{\mathbb{A}}(39/25))^{k},
\end{align*}
 for some $k\in\mathbb{Z}$ such that $|z|\asymp k$. We note that $\zeta_{\mathbb{A}}(39/25)=c$ for some constant $c$ so that 
 \[(\zeta_{\mathbb{A}}(39/25))^{k}\ll c^{\frac{\log|D|}{\log_2|D|\ln\log_2|D|}}=|D|^{\frac{\log c}{\log_2|D|\ln\log_2|D|}}\ll |D|^{1/42},\]
 for $n$ large enough. Hence we have the desired result.
  
\end{proof}
\subsection{Evaluating $S_1$: Contribution of the square terms.}\label{PropRandProd}
The last step is to understand the main term $S_1$. From Lemma \ref{orthogonality}  we have that 
\begin{multline*}S_1=\left(1+O\left(\frac1{(\log|D|)^{B}}\right)\right)\left(\sum_{\substack{f \text{ monic}\\ |f|\le |D|^{1/3}\\P|f \Rightarrow \deg P\le M}}\frac{d_z(f^2)}{|f|^2}\left(|\hn| \cdot \prod_{P \mid f} \left(1+\frac{1}{|P|}\right)^{-1}\right.\right. \\
\left.\left.+ O\left(\sqrt{|\hn|}\right)\right)\right).
\end{multline*}
Estimating this term is where the difficulties lie, thus enters the random model $L(1,\mathbb{X})$:
 Let $\{\mathbb{X}(P)\}$ denote a sequence of  independent random variables indexed by $P\in\mathbb{A}$ an irreducible (prime) element, which takes the values $0,\pm1$ described as \eqref{defrandvar}. The goal of this section is to prove the following Lemma. Theorem \ref{momLthm} follows immediately after combining Lemmas  \ref{relatechitoX} and  \ref{expf}.
 
 \begin{lem}\label{relatechitoX}
Let $D\in \hn$. Let $z\in \mathbb{C}$ be such that $|z|\le \frac{\log|D|}{260\log_2|D|\ln\log_2|D|}$.
 Then 
\begin{equation*}
\frac1{|\hn|}\sum_{D\in\hn}L(1,\chi_D)^z=\mathbb{E}(L(1,\mathbb{X})^z)\left(1+O\left(\frac1{(\log|D|)^{11}}\right)\right).
\end{equation*}
\end{lem}

The expectation of $\mathbb{X}$, $\mathbb{E}(\mathbb{X}(P))$, is zero and $\mathbb{E}(\mathbb{X}(P)^2)= \frac{|P|}{|P|+1}$. We extend the definition of $\mathbb{X}$ to all monic polynomials $f\in\mathbb{A}$ as in \eqref{extendeddefX}.
Then, since $\mathbb{X}$ is independent on the primes,
if $f=P_1^{e_1}P_2^{e_2}\cdots P_s^{e_s}$ we have
\[\mathbb{E}(\mathbb{X}(f))=\prod_{i=1}^{s}\mathbb{E}(\mathbb{X}(P_i)^{e_i}).\]
We note that $\mathbb{X}(P)^{e_j}=\mathbb{X}(P)$ if $e_j\equiv 1 (\mod 2)$ and $\mathbb{X}(P)^2$ if $e_j\equiv 0(\mod 2)$.
Combining this fact with the independence of $\mathbb{X}$ 
we see that 
\[\mathbb{E}(\mathbb{X}(P)^{e_j})=\begin{cases}
0 & \text{ if } e_j\equiv 1 \pmod 2\\
\frac{|P|}{|P|+1} & \text{ if }e_j\equiv 0 \pmod 2.
\end{cases}
\] 
Thus we have proved:
\begin{lem}\label{expf}
\[\mathbb{E}(\mathbb{X}(f))=\begin{cases}
0 & \text{ if } f \text{ is not a square}\\
\prod_{P|f}\left(1+\frac1{|P|}\right)^{-1}&\text{ if } f \text{ is a square}.
\end{cases}
\]
\end{lem}
Then, for any $z\in\mathbb{C}$ since $d_z(f)$ and $|f|$ are scalars we see that 
\[\mathbb{E}(L(1,\mathbb{X})^z)=\sum_{f\text{ monic}}\frac{d_z(f)\mathbb{E}(\mathbb{X}(f))}{|f|}=\sum_{f\text{ monic}}\frac{d_z(f^2)}{|f|^2}\prod_{P|f}\left(1+\frac1{|P|}\right)^{-1},\]
where $L(1,\mathbb{X})$ is defined in \eqref{randprod}. We recognize the shape of $S_1$ from this. 
On the other hand, from the random Euler product definition we have 
\begin{equation*}
\mathbb{E}(L(1,\mathbb{X})^z)=\prod_{P \text{ irreducible}}E_{P}(z),
\end{equation*}
where 
\begin{multline}\label{EPdef}
E_{P}(z):=\mathbb{E}\left(\left(1-\frac{\mathbb{X}(P)}{|P|}\right)^{-z}\right)=\frac1{|P|+1}+\frac{|P|}{2(|P|+1)}\left(\left(1-\frac1{|P|}\right)^{-z}+\left(1+\frac1{|P|}\right)^{-z}\right).
\end{multline}
Now, we notice if $\deg P>M$ then we can use the following Taylor expansions 
\[\left(1-\frac1{|P|}\right)^{-z}=1+\frac{z}{|P|}+O\left(\frac{|z|}{|P|^2}\right),\]
and
\[\left(1+\frac1{|P|}\right)^{-z}=1-\frac{z}{|P|}+O\left(\frac{|z|}{|P|^2}\right).\] 
That is to say, for $P$ irreducible and  $\deg P > M$ we have  $E_P(z)=1+O(|z|/|P|^2)$, so that 
\[\prod_{\substack{P\text{ irreducible}\\ \deg P>M}}E_P(z)\ll \exp\left(|z|\sum_{\deg P>M}\frac1{|P|^2}\right)=1+O\left(\frac1{(\log |D|)^{B}}\right),\]
this last equality follows from the relative sizes of $|z|$ and $M$, where we again note that $M=A\log_2|D|$ and we choose $A$ large enough to provide the desired error term above.

Finally, using \eqref{truncsumfprod} of Lemma \ref{truncationpart2} we see
\begin{multline}
\mathbb{E}(L(1,\mathbb{X})^z)=\sum_{\substack{f\text{ monic}\\P|f\Rightarrow \deg P\le M}}\frac{d_z(f^2)}{|f|^2}\prod_{P|f}\left(1+\frac1{|P|}\right)^{-1}\left(1+O\left(\frac1{(\log|D|)^{B}}\right)\right)\\
=\sum_{\substack{f\text{ monic}\\|f|<|D|^{1/3}\\P|f\Rightarrow \deg P\le M}}\frac{d_z(f^2)}{|f|^2}\prod_{P|f}\left(1+\frac1{|P|}\right)^{-1}\left(1+O\left(\frac1{(\log|D|)^{B}}\right)\right).
\end{multline}

The above discussion and taking the choice $A=26$ in Lemma \ref{keylemma} gives $B=11$ which proves Lemma \ref{relatechitoX}.  Using the fact that $|D|=q^n$ we obtain Theorem \ref{momLthm}. Corollary \ref{momhDthm} also follows from this discussion by simply scaling everything appropriately via \eqref{Artinform} and the fact that expectation is linear. Finally, Corollary \ref{momhDRD} is obtained in the same way but instead we apply \eqref{realArtin}.
\section{The distribution of values of $L(1, \mathbb{X})$}\label{distrandmod}
Here we aim to prove results about $\Phi_{\mathbb{X}}(\tau)$ and $\Psi_{\mathbb{X}}(\tau)$. The proofs of $\Psi_{\mathbb{X}}(\tau)$ require only minor adjustments to those for $\Phi_{\mathbb{X}}(\tau)$. The discussion in this section is modelled  after \cite[Section 4]{AY}. These authors use a saddle point analysis to achieve their results, and we adapt that idea here. To this end, we define some useful auxiliary functions.
For $z\in \mathbb{C}$ 
 define 
\begin{equation}\label{Ldefexp}
\mathcal{L}(z):=\ln\mathbb{E}(L(1,\mathbb{X})^z)=\sum_{P \text{ irreducible}}\ln(E_{P}(z)),
\end{equation}
where $E_{P}(z)$ is defined as in \eqref{EPdef}.
Furthermore we consider the equation
\begin{equation}\label{kappadef}
(\mathbb{E}(L(1,\mathbb{X})^r)(e^{\gamma}\tau)^{-r})'=0\Leftrightarrow \mathcal{L}'(r)=\ln(\tau)+\gamma,
\end{equation}
where the derivative is taken with respect to the real variable $r$. It follows from Proposition \ref{curlyLbound} that $\lim_{r\to\infty}\mathcal{L}'(r)=\infty$, one can easily check that $E_P''(r)E_P(r)>(E'_P(r))^2$ for all monic irreducible polynomials $P$, and thus $\mathcal{L}''(r)>0$. Therefore \eqref{kappadef} has a unique solution: we define $\kappa =\kappa(\tau)$ as this unique solution.

Finally, we define 
\begin{equation}\label{def-f}
f(t):=\begin{cases}
\ln\cosh (t) & \text{ if } 0\le t<1\\
\ln\cosh(t)-t & \text{ if } t\ge1.
\end{cases}
\end{equation}

\subsection{Distribution of the Random Model.}\label{randmoddist}
\begin{thm}\label{distnsig1X}
Let $\tau$ be large and $\kappa$ denote the unique solution to \eqref{kappadef}. Then, we have 
\begin{equation}\label{PhiXExp}
\Phi_{\mathbb{X}}(\tau)=\frac{\mathbb{E}(L(1,\mathbb{X})^{\kappa})(e^{\gamma}\tau)^{-\kappa}}{\kappa\sqrt{2\pi \mathcal{L}''(\kappa)}}\left(1+O\left(\sqrt{\frac{\log\kappa}{\kappa}}\right)\right).
\end{equation}
Moreover, for any $0\le\lambda\le1/\kappa$ we have 
\begin{equation}\label{PhiXExplambda}
\Phi_{\mathbb{X}}(e^{-\lambda}\tau)=\Phi_{\mathbb{X}}(\tau)(1+O(\lambda\kappa)). 
\end{equation}
\end{thm}
We prove Theorem \ref{distfuncX} from this and the following proposition which gives some estimates on the size of $\mathcal{L}$ and its first few derivatives.
\begin{pro}\label{curlyLbound}
Let $f$ be defined by \eqref{def-f}. Let $c_q\ge q$ be a positive constant depending on $q$ and let $k\in\mathbb{Z}$ be the unique positive integer such that $q^k\le r<q^{k+1}$ and let $t:=\frac{r}{q^k}$. 
With this notation in mind for r any real number such that $r\ge c_q$ we have 
\begin{equation}\label{curlyL}
\mathcal{L}(r)=r\left(\ln\log r +\gamma\right)+\frac{r}{\log r}G_1(t)+O\left(\frac{r\log\log r}{(\log r)^2}\right),
\end{equation}
where 
\begin{equation}\label{defG1}
G_1(t):=\frac12-\log t+\sum_{l=-\infty}^{\infty}\frac{f(tq^l)}{tq^l}.
\end{equation}
Furthermore, we have 
\begin{equation}\label{curlyLprime}
\mathcal{L}'(r)=\ln\log r+\gamma+\frac1{\log r}G_2(t)+O\left(\frac{\log\log r}{(\log r)^2}\right),
\end{equation}
where 
\begin{equation}\label{defG2}
G_2(t):=\frac12-\log t+\sum_{l=-\infty}^{\infty}f'(tq^l).
\end{equation}
Moreover, for all real numbers $y$ , $x$ such that $|y|\ge 3$ and for all $x$ such that $|y|\le |x|$ we have 
\begin{equation}\label{secondandthirdderiv}
\mathcal{L}''(y)\asymp \frac1{|y|\ln|y|} \text{ and }  \mathcal{L}'''(y)\ll \frac1{|y|^2\ln|y|}.
\end{equation}
 \end{pro}

Combining these results gives Theorem \ref{distfuncX}:
\begin{proof}[Proof of Theorem \ref{distfuncX}]
By Theorem \ref{distnsig1X} and \eqref{curlyLprime} we have 
\begin{align*}
\Phi_{\mathbb{X}}(\tau)&=\frac{\mathbb{E}(L(1,\mathbb{X})^{\kappa})(e^{\gamma}\tau)^{-\kappa}}{\kappa\sqrt{2\pi \mathcal{L}''(\kappa)}}\left(1+O\left(\sqrt{\frac{\log\kappa}{\kappa}}\right)\right)\\
&=\exp\left(\mathcal{L}(\kappa)-\kappa(\ln \tau+\gamma)+O(\log\kappa)\right),
\end{align*}
where $\kappa$ is the unique solution which satisfies \eqref{kappadef}.\\ 
Also from  \eqref{curlyLprime} we have 
\begin{equation}\label{logtausize}
\ln\tau=\ln\log\kappa+\frac{G_2(q^{\{\log\kappa\}})}{\log\kappa}+O\left(\frac{\log\log\kappa}{(\log\kappa)^2}\right).
\end{equation}
Hence using \eqref{curlyL} we obtain 
\[\Phi_{\mathbb{X}}(\tau)=\exp\left(\kappa\frac{G_1(q^{\{\log\kappa\}})-G_2(q^{\{\log\kappa\}})}{\log\kappa}+O\left(\frac{\kappa\log\log\kappa}{(\log\kappa)^2}\right)\right).\]
We note that from \eqref{logtausize} we have $\kappa\asymp q^{\tau}$ and thus
\[\log\kappa=\tau-G_2(q^{\{\log\kappa\}})+O\left(\frac{\log\tau}{\tau}\right),\]
since for every $\tau$ there is a unique $\kappa$ which satisfies \eqref{kappadef} and  $G_2(q^{\{\log\kappa\}})$ is bounded for any $\kappa$.
This is enough to obtain the shape of the result. It remains to prove that
\[-\frac1{\ln q}+\frac{\ln\cosh c}{c}-\tanh c<G_1(q^{\{\log\kappa\}})-G_2(q^{\{\log\kappa\}})<\frac{\ln(\cosh(q))}{q}-\tanh(q),\]
where $c=1.28377...$
For ease of notation let $t=q^{\{\log\kappa\}}$, we note that $1\le t<q$ and 
\[G_1(t)-G_2(t)=\sum_{l=-\infty}^{\infty}\left(\frac{f(tq^l)}{tq^l}-f'(tq^l)\right).\]
We recall from the definition of $f$ that the shape is different depending on the size of the input. So we split the sum: 
\begin{align*}
G_1(t)-G_2(t)&=\sum_{l<-\log t}\left(\frac{\ln\cosh(tq^l)}{tq^l}-\tanh(tq^l)\right)+\sum_{l\ge -\log t}\left(\frac{\ln\cosh(tq^l)-tq^l}{tq^l}-(\tanh(tq^l)-1)\right)\\
&=\sum_{l=-\infty}^{\infty}\left(\frac{\ln\cosh(tq^l)}{tq^l}-\tanh(tq^l)\right).
\end{align*}
To prove the upper bound, it is enough to show that all the summands are negative and so the sum will be less than the contribution from the $l=0$ term. We note that 
\[\frac{d}{dy}\left[\frac{\ln\cosh y}{y}-\tanh y\right]=\frac{\tanh y}y-\frac{\ln\cosh y}{y^2}-\text{sech}^2y=0\] 
when $y=\pm1.28377\ldots$ but the argument of our function is $tq^l>0$ for all $l$ since $t,q>0$, so we need only consider $c=1.28377\ldots$. A simple calculation shows that  this is a minimum and that $\frac{\ln\cosh y}{y}-\tanh y$ is strictly decreasing on the interval $(0,c)$ and strictly increasing on the interval $(c,\infty)$. Taking the limit as $y\to 0$ and $y\to \infty$ we see these are both $0$, hence all of the summands are negative since $\frac{\ln\cosh c}{c}-\tanh c=-0.339834\ldots$. Therefore we have a suitable upper bound by simply evaluating $\frac{\ln\cosh tq^l}{tq^l}-\tanh tq^l$ at $l=0$, which gives $\frac{\ln\cosh t}{t}-\tanh t$. As we discussed this function reaches its maximum value when $t$ does, in this case $t<q$.  \\
In order to consider the lower bound, we note that 
\[\int_{-\infty}^{\infty}\frac{\ln\cosh tq^y}{tq^y}-\tanh tq^y dy=-\frac1{\ln q}\le\sum_{l=-\infty}^{\infty}\left(\frac{\ln\cosh tq^l}{tq^l}-\tanh tq^l\right)-\left(\frac{\ln\cosh t}{t}-\tanh t\right), \]
hence 
\[\sum_{l=-\infty}^{\infty}\frac{\ln\cosh tq^l}{tq^l}-\tanh tq^l\ge -\frac1{\ln q}+\frac{\ln\cosh t}{t}-\tanh t.\]
Finally, as we discussed this is minimized when $t=c$.
\end{proof}
 \subsubsection{Tools for proving Proposition \ref{curlyLbound}}
First we recall the following standard estimates on $f$ and $f'$. 
\begin{lem}\label{fbounds}\cite[Lemma 4.5]{Lam15} $f$ is bounded on $[0,\infty)$ and $f(t)=t^2/2+O(t^4)$ if $0\le t<1$. Moreover, we have 
\[f'(t)=\begin{cases}
t+O(t^2)& \text{ if }0<t <1\\
O(e^{-2t})& \text{ if }t\ge1.
\end{cases}
\]
\end{lem}

\begin{lem}\label{logEpreal}
Let $r\ge c_q$ be a real number, where $c_q$ is a positive constant depending on $q$. Then we have 
\begin{equation}\label{logEP}
\ln E_P(r)=\begin{cases}
-r\ln(1-1/|P|)+O(1) & \text{ if }|P|\le r^{2/3},\\
\ln\cosh\left(\frac{r}{|P|}\right)+O\left(\frac{r}{|P|^2}\right) & \text{ if } |P|>r^{2/3},
\end{cases}
\end{equation}
and 
\begin{equation}\label{logderEP}
\frac{E_P'}{E_P}(r)=\begin{cases}
-\ln\left(1-\frac1{|P|}\right)\left(1+O\left(e^{-r^{1/3}}\right)\right) & \text{ if } |P|\le r^{2/3}\\
\frac1{|P|}\tanh\left(\frac{r}{|P|}\right)+O\left(\frac1{|P|^2}+\frac{r}{|P|^3}\right) & \text{ if } |P|>r^{2/3}.
\end{cases}
\end{equation}

\end{lem}

\begin{proof}
First we prove \eqref{logEP}. Start by considering $|P|\le r^{2/3}$. Since $|P|$ is small we have
\begin{multline}\label{EPestpsmall}
E_P(r)=\frac{|P|}{2(|P|+1)}\left(1-\frac1{|P|}\right)^{-r}\left(1+\left(1+\frac1{|P|}\right)^{-r}+\frac2{|P|}\left(1-\frac1{|P|}\right)^r\right)\\
=\frac{|P|}{2(|P|+1)}\left(1-\frac1{|P|}\right)^{-r}\left(1+O(\exp(-r^{1/3}))\right),
\end{multline}
taking logs gives the desired result. 

Suppose now that $|P| >r^{2/3}$, we see that 
\[\left(1-\frac1{|P|}\right)^{-r}=e^{r/|P|}\left(1+O\left(\frac{r}{|P|^2}\right)\right)\text{ and } \left(1+\frac1{|P|}\right)^{-r}=e^{-r/|P|}\left(1+O\left(\frac{r}{|P|^2}\right)\right),\]
thus using the bounds $\cosh(t)-1\ll t\cosh(t)$ and $\sinh(t)\ll t\cosh(t)$, which are valid for all $t\ge 0$ we see that
\begin{align} \label{EPestpbig}
\nonumber E_P(r)&=\frac{|P|}{|P|+1}\cosh\left(\frac{r}{|P|}\right)\left(1+O\left(\frac{r}{|P|^2}\right)\right)+\frac1{|P|}\\
&=\cosh\left(\frac{r}{|P|}\right)\left(1+O\left(\frac{r}{|P|^2}\right)\right).
\end{align}
Taking logs completes the proof.\\
For \eqref{logderEP}, we first see from \eqref{EPdef} that 
\[E'_P(r)=\frac{-|P|}{2(|P|+1)}\left(\left(1-\frac1{|P|}\right)^{-r}\ln\left(1-\frac1{|P|}\right)+\left(1+\frac1{|P|}\right)^{-r}\ln\left(1+\frac1{|P|}\right)\right).\]
If $|P|\le r^{2/3}$ then  \eqref{EPestpsmall} finishes the claim.\\
On the other hand for $|P|>r^{2/3}$ we have
\begin{multline*}
E'_P(r)=\frac{|P|}{2(|P|+1)}\left(\frac{e^{r/|P|}}{|P|}-\frac{e^{-r/|P|}}{|P|}\right)\left(1+O\left(\frac1{|P|}+\frac{r}{|P^2|}\right)\right)\\
=\frac1{|P|}\sinh\left(\frac{r}{|P|}\right)\left(1+O\left(\frac1{|P|^2}+\frac{r}{|P|^3}\right)\right)+O\left(\frac1{|P|^2}\cosh\left(\frac{r}{|P|}\right)\right).
\end{multline*}
Combining this with \eqref{EPestpbig} we have the desired result.
\end{proof}
\begin{proof}[Proof of Proposition \ref{curlyLbound}.]
For the entire proof, we recall that $k\in\mathbb{Z}$ is the unique positive integer such that $q^k\le r<q^{k+1}$ and let $t:=\frac{r}{q^k}$.\\

We first prove the result for $\mathcal{L}(r)$.
By Lemma \ref{logEpreal} and Lemma \ref{fbounds} we have 
\begin{align}
\nonumber\mathcal{L}(r)&=-r\sum_{|P|\le r^{2/3}}\ln\left(1-\frac1{|P|}\right)+\sum_{|P|>r^{2/3}}\ln\cosh\left(\frac{r}{|P|}\right)+O(r^{2/3})\\
&=-r\sum_{\deg P\le k}\ln\left(1-\frac1{|P|}\right)+\sum_{|P|>r^{2/3}}f\left(\frac{r}{|P|}\right)+O(r^{2/3}).
\end{align}
The first summand is taken care of by recognizing Mertens' theorem, which we will apply at the end. The more interesting part of the proof comes from the second sum. First, from  the prime number theorem we get
\begin{align*}
\sum_{|P|>r^{2/3}}f\left(\frac{r}{|P|}\right)=\sum_{n>\frac23\log  r}\frac{q^n}nf\left(\frac{r}{q^n}\right)+O\left(\sum_{n>\frac23\log  r}\frac{q^{n/2}}nf\left(\frac{r}{q^n}\right)\right).
\end{align*}
The error term is 
 \begin{align*}
 \sum_{n>\frac23\log  r}\frac{q^{n/2}}nf\left(\frac{r}{q^n}\right)&\ll\sum_{n\ge\log  r}q^{n/2}\frac{r^2}{q^{2n}}+\sum_{\frac23\log  r<n<\log r}q^{n/2}\text{ by Lemma \ref{fbounds}},\\
 &\ll \sqrt{r}.
 \end{align*}

It remains to consider 
\begin{multline*}
\sum_{n>\frac23\log  r}\frac{q^n}nf\left(\frac{r}{q^n}\right)=\sum_{n>k+\log  k}+\sum_{n<k-\log  k}+\sum_{k-\log  k\le n\le k+\log  k}\left(\frac{q^n}nf\left(\frac{r}{q^n}\right)\right)\\
=T_1+T_2+T_3.
\end{multline*}
We first bound $T_1$ and $T_2$, again referring to Lemma \ref{fbounds}
\begin{align*}
T_1&\ll \frac1k\sum_{n>k+\log  k}\frac{q^nr^2}{q^{2n}}\ll \frac{r^2}{k}\sum_{n>k+\log  k}\frac1{q^n}\\
&\ll\frac{r^2}{k}\frac1{q^{k+\log  k}}\ll \frac{r}{(\log  r)^2} \text{ by the choice of } k \text{ with respect to } r,
\end{align*}
and similarly
\begin{align*}
T_2&\ll \sum_{n<k-\log  k}\frac{q^n}{n}\ll \frac{q^{k-\log  k}}{k}\ll \frac{r}{(\log r)^2}.
\end{align*}
For $T_3$ we notice that for $n\in [k-\log  k,k+\log  k]$, we have $\frac1n=\frac1k\left(1+O\left(\frac{\ln k}{k}\right)\right)$. Using this, we factor out $\frac{q^k}{k}$ and do the variable change $l=k-n$ so that 
\[T_3=\left(\frac{q^k}{k}+O\left(\frac{q^k \ln k}{k^2}\right)\right)\sum_{|l|\le \log  k}\frac{f(tq^l)}{q^l},\]
where we recall $t:=\frac{r}{q^k}$. Next, we see that for $|l|>\log  k$ the sum is small: 
\[\sum_{|l|>\log  k}\frac{f(tq^l)}{q^l}\ll \sum_{l>\log  k}\frac1{q^l}+\sum_{l<-\log  k}q^l\ll \frac1{q^{\ln k}}\ll\frac1k.\]
Hence we have 
\[ T_3=\frac{q^k}k\left(\sum_{l=-\infty}^{\infty}\frac{f(tq^l)}{q^l}+O\left(\frac{\ln k}{k}\right)\right).\]
Returning this to an expression in terms of $r$ we see
\[ T_3=\frac{r}{\log  r}\left(\sum_{l=-\infty}^{\infty}\frac{f(tq^l)}{tq^l}+O\left(\frac{\ln \ln r}{\ln r}\right)\right).\]
Finally, we complete the bound of $\mathcal{L}(r)$ by applying Mertens' theorem to the first summand and convert everything in terms of $r$.  Combining the terms which have $\frac{r}{\log  r}$ in common we achieve the claimed result.\\

For  $\mathcal{L}'(r)$, we again appeal to Lemma \ref{logEpreal} and Lemma \ref{fbounds} giving 
\begin{align}
\nonumber\mathcal{L}'(r)&=-\sum_{|P|\le r^{2/3}}\ln\left(1-\frac1{|P|}\right)+\sum_{|P|>r^{2/3}}\frac{\tanh\left(\frac{r}{|P|}\right)}{|P|}+O(r^{-1/3})\\
&=-\sum_{\deg P\le k}\ln\left(1-\frac1{|P|}\right)+\sum_{|P|>r^{2/3}}\frac{f'\left(\frac{r}{|P|}\right)}{|P|}+O(r^{-1/3}).
\end{align}
Applying the prime number theorem to the second sum we obtain 
\begin{equation*}
\sum_{|P|>r^{2/3}}\frac{f'\left(\frac{r}{|P|}\right)}{|P|}=\sum_{n>2/3\log r}\frac{f'(r/q^n)}{n}+O\left(\sum_{n>2/3\log r}\frac{f'(r/q^n)}{q^{n/2}n}\right).
\end{equation*}
The error term in this case is 
\begin{align*}
\sum_{n>2/3\log r}\frac{f'(r/q^n)}{q^{n/2}n}&\ll\sum_{n\ge\log r}\frac{r}{q^{3n/2}n}+\sum_{2/3\log r<n<\log r}\frac{e^{-2r/q^n}}{q^{n/2}n} \text{ by Lemma \ref{fbounds}}\\
&\ll \frac1{r^{1/3}\log r}.
\end{align*}
As before, we split the remaining sum into 3 pieces: 
\begin{multline*}
\sum_{n>2/3\log r}\frac{f'(r/q^n)}{n}=\sum_{n>k+\log  k}+\sum_{n<k-\log  k}+\sum_{k-\log  k\le n\le k+\log  k}\left(f'\left(\frac{r}{q^n}\right)\frac1n\right)\\
=T'_1+T_2'+T_3'.
\end{multline*}
We first bound $T_1'$ and $T_2'$, referring to Lemma \ref{fbounds} gives
\begin{align*}
T_1'&\ll\sum_{n>k+\log  k}\frac{r}{q^nn}\ll \frac{r}k\sum_{n>k+\log  k}\frac1{q^n}\\
&\ll\frac{r}{k}\frac1{q^{k+\log  k}}\ll\frac1{(\log r)^2} \text{by the choice of $k$ with respect to $r$,}
\end{align*}
similarly
\begin{equation*}
T_2'\ll \sum_{n<k-\log  k}\frac{e^{-2r/q^n}}{n}\ll \frac1{(\log r)^2}.
\end{equation*}
For $T_3'$ we notice that for $n\in [k-\log  k,k+\log  k]$, we have $\frac1n=\frac1k\left(1+O\left(\frac{\ln k}{k}\right)\right)$. Using this, we factor out $\frac{q^k}{k}$ and do the variable change $l=k-n$ and recall $t:=\frac{r}{q^k}$, so that 
\[T_3'\ll \left(\frac1{k}+O\left(\frac{\ln k}{k^2}\right)\right)\sum_{|l|<\log  k}f'(tq^l).\]
Next, we show this sum is small for $|l|>\log  k$:
\begin{align*}
\sum_{|l|>\log  k}f'(tq^l)&\ll \sum_{l>\log  k}e^{-2tq^l}+\sum_{l<-\log  k}tq^l\\
&\ll e^{-2\log  r}+\frac1{q^{\log  k}}, \text{ the bound on the second sum follows from } 1\le t\le q\\
&\ll \frac1{\log  r}.
\end{align*}
Hence we have 
\[T_3'\ll\frac1{k}\left(\sum_{l=-\infty}^{\infty}f'(tq^l)+O\left(\frac{\ln k}{k}\right)\right).\]
Finally, we complete the bound of $\mathcal{L}'(r)$ by applying Mertens' theorem to the first summand and convert everything in terms of $r$.  Combining the terms which have $\frac{1}{\log  r}$ in common we achieve the claimed result.\\
\end{proof}

\subsubsection{Proof of Theorem \ref{distnsig1X}} 
One of the key ingredients in the proof of Theorem \ref{distnsig1X} is to show that $|\mathbb{E}(L(1,\mathbb{X})^{r+it})|/\mathbb{E}(L(1,\mathbb{X})^{r})$ is rapidly decreasing in $t$ when $|t|\ge\sqrt{r\ln r}$. For this we prove the following lemmas.\\

\begin{lem}\label{ratioEp}
Let $r$ is a large positive number and $c_q\ge q$ a positive constant depending on $q$. If $|P|>\frac{r}{c_q}$, then for some positive constant $b_1$ we have 
\[\frac{|E_P(r+it)|}{E_P(r)}\le \exp\left(-b_1\left(1-\cos\left(t\ln\left(\frac{|P|+1}{|P|-1}\right)\right)\right)\right),\]
where $c_q$ is a positive constant dependent on $q$.
\end{lem}

\begin{proof}
Let $x_1$, $x_2$ and $x_3$ be positive real numbers and $\theta_2$ and $\theta_3$ be real numbers. We use the following inequality established in the proof of \cite[Lemma 3.2]{GranSound}:
\[|x_1+x_2e^{i\theta_2}+x_3e^{i\theta_3}|\le (x_1+x_2+x_3)\exp\left(-\frac{x_1x_3(1-\cos \theta_3)}{(x_1+x_2+x_3)^2}\right).\]
Choosing $x_1=\frac{|P|}{2(|P|+1)}(1+1/|P|)^{-r}$, $x_2=\frac1{|P|+1}$ and $x_3=\frac{|P|}{2(|P|+1)}(1-1/|P|)^{-r}$ with $\theta_2=t\ln(1+1/|P|)$ and $\theta_3=t\ln\left(\frac{|P|+1}{|P|-1}\right)$ provides the desired result since $|P|>\frac{r}{c_q}$. 
\end{proof}
\begin{lem}\label{decayexpect}
Let $r$ be large and let $c_q\ge q>4$ be a positive constant dependent on $q$. Then there exists a constant $b_2>0$ such that 
\[\frac{|\mathbb{E}(L(1,\mathbb{X})^{r+it})|}{\mathbb{E}(L(1,\mathbb{X})^r)}\ll 
\begin{cases}
\exp\left(-b_2\frac{t^2}{r \ln r}\right) & \text{ if } |t|\le \frac{r}{c_q}\\
\exp\left(-b_2\frac{|t|}{ \ln |t|}\right) & \text{ if } |t|>\frac{r}{c_q}.
\end{cases}
\]
\end{lem}
\begin{proof}
Let $z=r+it$. Since $|E_P(z)|\le E_P(r)$ we obtain for any real numbers $q\le y_1<y_2$ 
\begin{equation}
\frac{|\mathbb{E}(L(1,\mathbb{X})^{z})|}{\mathbb{E}(L(1,\mathbb{X})^r)}\le \prod_{y_1\le |P|\le y_2}\frac{|E_P(z)|}{E_P(r)}.
\end{equation}
Note that $|t|\ln\left(\frac{|P|+1}{|P|-1}\right)\sim2|t|/|P|$ so that when $|t|\le \frac{|P|}{c_q}$ we have 
\[1-\cos\left(|t|\ln\left(\frac{|P|+1}{|P|-1}\right)\right)\gg\frac{|t|^2}{|P|^2}.\]
If $|t|\le \frac{r}{c_q}$ then, we choose $y_1=r$ and $y_2=c_qr/2$. Appealing to Lemma \ref{ratioEp} we have 
\begin{align*}
\prod_{y_1\le |P|\le y_2}\frac{|E_P(z)|}{E_P(r)}&\ll \prod_{\ln r\le d\le\ln(c_q r/2)+1} \exp\left(-b_1\frac{q^d}{2d}\frac{|t|^2}{q^{2d}}\right)\\
&=\exp\left(-\frac{b_1|t|^2}{2}\sum_{\ln r\le d\le\ln(c_qr/2)+1}\frac1{dq^d}\right)\ll\exp\left(-b_2\frac{|t|^2}{r\ln r}\right).
\end{align*}
In the case of $|t|>\frac{r}{c_q}$ we use a similar argument but choose $y_1=c_q|t|$ and $y_2=2c_q|t|$ to complete the result.
\end{proof}
 Let $\varphi(y)=1$ if $y>1$ and equal to $0$ otherwise. Then we have the following smooth analogue of Perron's formula: 

\begin{lem}\cite[Lemma 4.7]{AY}
\label{PerronLemma}
Let $\lambda>0$ be a real number and $N$ be a positive integer. For any $c>0$ we have for $y>0$
\begin{equation}\label{Perron1}
0\le \frac1{2\pi i}\int_{c-i\infty}^{c+i\infty}y^s\left(\frac{e^{\lambda s}-1}{\lambda s}\right)^{N}\frac{ds}s-\varphi(y)\le\frac1{2\pi i}\int_{c-i\infty}^{c+i\infty}y^s\left(\frac{e^{\lambda s}-1}{\lambda s}\right)^{N}\frac{1-e^{-\lambda Ns}}{s}ds, 
\end{equation}
and 
\begin{equation}\label{Perron2}
0\le\varphi(e^{\lambda}y)-\varphi(y)\le \frac1{2\pi i}\int_{c-i\infty}^{c+i\infty}y^s\left(\frac{e^{\lambda s}-1}{\lambda s}\right)\frac{e^{\lambda s}-e^{-\lambda s}}{s}ds.
\end{equation}
\end{lem}

\begin{proof}[Proof of Theorem \ref{distnsig1X}]
We first prove \eqref{PhiXExp}. 
Let $0<\lambda<1/(2\kappa)$ be a real number which we choose later. Using \eqref{Perron1} from  Lemma \ref{PerronLemma}, taking $N=1$ we obtain 
\begin{align}\label{PhiXdiff}
\nonumber 0&\le \int_{\kappa-i\infty}^{\kappa+i\infty}\mathbb{E}(L(1,\mathbb{X})^s)(e^{\gamma}\tau)^{- s}\frac{e^{\lambda s}-1}{\lambda s}\frac{ds}{s}-\Phi_{\mathbb{X}}(\tau)\\
&\le \int_{\kappa-i\infty}^{\kappa+i\infty}\mathbb{E}(L(1,\mathbb{X})^s)(e^{\gamma}\tau)^{- s}\frac{(e^{\lambda s}-1)}{\lambda s}\frac{(1-e^{-\lambda s})}{s}ds.
\end{align}
Since $\lambda\kappa<1/2$ we have $|e^{\lambda s}-1|\le 3$ and $|e^{-\lambda s}-1|\le 2$. Hence, using Lemma \ref{decayexpect} along with the fact that $|\mathbb{E}(L(1,\mathbb{X})^s)|\le\mathbb{E}(L(1,\mathbb{X})^{\kappa})$ we obtain, for some constant $b_3>0$ that 
\begin{equation}\label{perronbound1}
\int_{\kappa-i\infty}^{\kappa-i\kappa^{3/5}}+\int_{\kappa+i\kappa^{3/5}}^{\kappa+i\infty}\mathbb{E}(L(1,\mathbb{X})^s)(e^{\gamma}\tau)^{- s}\frac{e^{\lambda s}-1}{\lambda s}\frac{ds}{s} \ll\frac{e^{-b_3\kappa^{1/6}}}{\lambda\kappa^{3/5}}\mathbb{E}(L(1,\mathbb{X})^{\kappa})(e^{\gamma}\tau)^{-\kappa},
\end{equation}
and similarly, 
\begin{equation}\label{perronbound2}
\int_{\kappa-i\infty}^{\kappa-i\kappa^{3/5}}+\int_{\kappa+i\kappa^{3/5}}^{\kappa+i\infty}\mathbb{E}(L(1,\mathbb{X})^s)(e^{\gamma}\tau)^{- s}\frac{(e^{\lambda s}-1)}{\lambda s}\frac{(1-e^{-\lambda s})}{s}ds \ll\frac{e^{-b_3\kappa^{1/6}}}{\lambda\kappa^{3/5}}\mathbb{E}(L(1,\mathbb{X})^{\kappa})(e^{\gamma}\tau)^{-\kappa}.
\end{equation}
Let $s=\kappa+it$. If $|t|\le\kappa^{3/5}$ then $|(e^{\lambda s}-1)(1-e^{-\lambda s})|\ll \lambda^2|s|^2$, hence the remaining part of the integral is bounded as follows
\[\int_{\kappa-i\kappa^{3/5}}^{\kappa+i\kappa^{3/5}}\mathbb{E}(L(1,\mathbb{X})^s)(e^{\gamma}\tau)^{- s}\frac{(e^{\lambda s}-1)}{\lambda s}\frac{(1-e^{-\lambda s})}{s}ds\ll \lambda \kappa^{3/5}\mathbb{E}(L(1,\mathbb{X})^{\kappa})(e^{\gamma}\tau)^{-\kappa}.\]
Combining this estimate with \eqref{PhiXdiff}, \eqref{perronbound1} and \eqref{perronbound2} we obtain 
\begin{align}\label{PhiXbnd2}
\nonumber \Phi_{\mathbb{X}}(\tau)-\frac1{2\pi i}&\int_{\kappa-i\kappa^{3/5}}^{\kappa+i\kappa^{3/5}}\mathbb{E}(L(1,\mathbb{X})^s)(e^{\gamma}\tau)^{- s}\frac{(e^{\lambda s}-1)}{\lambda s^2}ds\\
&\ll \left(\lambda\kappa^{3/5}+\frac{e^{-b_3\kappa^{1/6}}}{\lambda\kappa^{3/5}}\right)\mathbb{E}(L(1,\mathbb{X})^{\kappa})(e^{\gamma}\tau)^{-\kappa}.
\end{align}
On the other hand we have from \eqref{secondandthirdderiv} when $|t|\le \kappa^{3/5}$ then
\[\mathcal{L}(\kappa+it)=\mathcal{L}(\kappa)+it\mathcal{L}'(\kappa)-\frac{t^2}2\mathcal{L}''(\kappa)+O\left(\frac{|t|^3}{\kappa^2\ln\kappa}\right).\]
We also note that
\[\frac{e^{\lambda s}-1}{\lambda s^2}=\frac1s(1+O(\kappa))=\frac1{\kappa}\left(1-i\frac{t}{\kappa}+O\left(\lambda\kappa+\frac{t^2}{\kappa^2}\right)\right).\]
Hence, using the fact that $\mathbb{E}(L(1,\mathbb{X})^s)=\exp(\mathcal{L}(s))$ and $\mathcal{L}'(\kappa)=\ln \tau+\gamma$ we find
\begin{align*}
&\mathbb{E}(L(1,\mathbb{X})^s)(e^{\gamma}\tau)^{- s}\frac{(e^{\lambda s}-1)}{\lambda s^2}\\
&=\frac1{\kappa}\mathbb{E}(L(1,\mathbb{X})^{\kappa})(e^{\gamma}\tau)^{- \kappa}\exp\left(-\frac{t^2}2\mathcal{L}''(\kappa)\right)\left(1-i\frac{t}{\kappa}+O\left(\lambda\kappa+\frac{t^2}{\kappa^2}+\frac{|t|^3}{\kappa^2\ln\kappa}\right)\right).
\end{align*}
Thus, since we have chosen $\kappa$ such that the integral involving $it/\kappa$ vanishes we have 
\begin{multline}\label{almostGuassian}
\frac1{2\pi i}\int_{\kappa-i\kappa^{3/5}}^{\kappa+i\kappa^{3/5}}\mathbb{E}(L(1,\mathbb{X})^s)(e^{\gamma}\tau)^{- s}\frac{(e^{\lambda s}-1)}{\lambda s^2}ds\\
=\frac1{\kappa}\mathbb{E}(L(1,\mathbb{X})^{\kappa})(e^{\gamma}\tau)^{- \kappa}\frac1{2\pi}\int_{-\kappa^{3/5}}^{\kappa^{3/5}}\exp\left(-\frac{t^2}2\mathcal{L}''(\kappa)\right)\left(1+O\left(\lambda\kappa+\frac{t^2}{\kappa^2}+\frac{|t|^3}{\kappa^2\ln\kappa}\right)\right)dt.
\end{multline}
Further, from \eqref{secondandthirdderiv} we have $\mathcal{L}''(\kappa)\asymp1/(\kappa\ln\kappa)$, so there exists a positive constant $b_4$ such that 
\[\frac1{2\pi}\int_{-\kappa^{3/5}}^{\kappa^{3/5}}\exp\left(-\frac{t^2}2\mathcal{L}''(\kappa)\right)dt =\frac1{\sqrt{2\pi\mathcal{L}''(\kappa)}}\left(1+O\left(e^{-b_4\kappa^{1/6}}\right)\right),\]
and
\begin{multline*}
\frac1{2\pi}\int_{-\kappa^{3/5}}^{\kappa^{3/5}}|t|^n\exp\left(-\frac{t^2}2\mathcal{L}''(\kappa)\right)dt\le\frac1{2\pi}\int_{-\infty}^{\infty}|t|^n\exp\left(-\frac{t^2}2\mathcal{L}''(\kappa)\right)dt \\
\ll \frac1{(\mathcal{L}''(\kappa))^{(n+1)/2}}\ll \frac{(\kappa\ln\kappa)^{n/2}}{\sqrt{2\pi\mathcal{L}''(\kappa)}}.
\end{multline*}
Inserting these estimates into \eqref{almostGuassian} we get 
\begin{multline}\label{expecbnd}
\frac1{2\pi i}\int_{\kappa-i\kappa^{3/5}}^{\kappa+i\kappa^{3/5}}\mathbb{E}(L(1,\mathbb{X})^s)(e^{\gamma}\tau)^{- s}\frac{(e^{\lambda s}-1)}{\lambda s^2}ds\\
=\frac{\mathbb{E}(L(1,\mathbb{X})^{\kappa})(e^{\gamma}\tau)^{- \kappa}}{\kappa\sqrt{2\pi \mathcal{L}''(\kappa)}}\left(1+O\left(\lambda\kappa+\sqrt{\frac{\ln\kappa}{\kappa}}\right)\right).
\end{multline}
Finally, combining the estimates \eqref{PhiXbnd2}, \eqref{expecbnd} and choosing $\lambda=\kappa^{-2}$ we obtain the desired result.\\
\indent Next we prove \eqref{PhiXExplambda}. To do this let $0\le \lambda\le 1/\kappa$.  Using \eqref{Perron2} from Lemma \ref{PerronLemma}, we have 
\[\varphi(e^{-\lambda}\tau)-\varphi(\tau)\le \frac1{2\pi i}\int_{\kappa-i\infty}^{\kappa+i\infty}\mathbb{E}(L(1,\mathbb{X})^s)(e^{\gamma}\tau)^{-s}\frac{(e^{\lambda s}-1)}{\lambda s}\frac{e^{\lambda s}-e^{-\lambda s}}{s}ds.\]
We write $s=\kappa+it$ and split this integral into two pieces: $|t|\le \lambda\sqrt{\kappa\ln\kappa}$ and $|t|>\lambda\sqrt{\kappa\ln\kappa}$. \\
\indent We note that both $|(e^{\lambda s}-1)/\lambda s|$ and $|(e^{\lambda s}-e^{-\lambda s})/\lambda s|$ are less than $4$. Therefore, it follows that the first part of the integral contributes $\ll \lambda \sqrt{\kappa\ln \kappa}\mathbb{E}(L(1,\mathbb{X})^{\kappa})(e^{\lambda}\tau)^{-\kappa}$. Then, from Lemma \ref{decayexpect} the second portion contributes
\begin{multline*}
\ll \lambda\mathbb{E}(L(1,\mathbb{X})^{\kappa})(e^{\lambda}\tau)^{-\kappa}\left(\int_{\sqrt{\kappa\ln\kappa}<|t|\le\tfrac{\kappa}{c_q}}e^{-b_2t^2/(\kappa\ln\kappa)} 
+\int_{|t|\ge \tfrac{\kappa}{c_q}}e^{-b_2|t|/(\ln|t|)}\right) \\
\ll \lambda \sqrt{\kappa\ln \kappa}\mathbb{E}(L(1,\mathbb{X})^{\kappa})(e^{\lambda}\tau)^{-\kappa}.
\end{multline*}
The final result follows from \eqref{PhiXExp} and \eqref{secondandthirdderiv}, specifically they prove: 
\begin{equation}\label{PhiXasymp}
\Phi_{\mathbb{X}}(\tau)\asymp \frac{\mathbb{E}(L(1,\mathbb{X})^{\kappa})(e^{\gamma}\tau)^{-\kappa}}{\kappa\sqrt{\mathcal{L}''(\kappa)}}\asymp \sqrt{\frac{\ln \kappa}{\kappa}}\mathbb{E}(L(1,\mathbb{X})^{\kappa})(e^{\gamma}\tau)^{-\kappa}.
\end{equation}

\end{proof}

\section{Proofs of Theorem \ref{distfuncchi} and Corollary \ref{disthD} }\label{MainThmProofs}
We begin with some notation: 
Let 
\[\mathbb{P}(L(1,\chi_D)>e^{\gamma}\tau):=\frac1{|\hn|}|\{D\in\hn : L(1,\chi_D)>e^{\gamma}\tau\}|\]
and 
\[M(z):=\frac1{|\hn|}\sum_{D\in\hn}L(1,\chi_D)^z.\]
\begin{proof}[Proof of Theorem \ref{distfuncchi}]
As in section \ref{randmoddist} let $\kappa=\kappa(\tau)$ be the unique solution to \eqref{kappadef}. Let $N$ be a positive integer and $0<\lambda<\min\{1/(2\kappa),1/N\}$ be a real value which we choose later. Finally, let $Y=b\log |D|/(\log_2|D|\log_3|D|)$ for some $b>0$ small enough. 

If $\log|D|$ is large enough, then for our range of $\tau$ we have $\kappa\le Y$, which follows from \eqref{logtausize}. Additionally, this means Lemma \ref{relatechitoX} holds for all $s=\kappa+it$ as long as $|t|\le Y$ so we consider the following integrals:
\[J(\tau)=\frac1{2\pi i}\int_{\kappa-i\infty}^{\kappa+i\infty}\mathbb{E}(L(1,\mathbb{X})^{s})(e^{\gamma}\tau)^{-s}\left(\frac{e^{\lambda s}-1}{\lambda s}\right)^N\frac{ds}{s},\]
and 
\[J_M(\tau)=\frac1{2\pi i}\int_{\kappa-i\infty}^{\kappa+i\infty}M(s)(e^{\gamma}\tau)^{-s}\left(\frac{e^{\lambda s}-1}{\lambda s}\right)^N\frac{ds}{s}.\]
By Lemma \ref{PerronLemma} we see that 
\begin{equation}\label{boundJ}
\Phi_{\mathbb{X}}(\tau)\le J(\tau)\le \Phi_{\mathbb{X}}(e^{-\lambda N}\tau)
\end{equation}
and 
\begin{equation}\label{boundJM}
\mathbb{P}(L(1,\chi_D)>e^{\gamma}\tau)\le J_M(\tau)\le \mathbb{P}(L(1,\chi_D)>e^{\gamma-\lambda N}\tau).
\end{equation}
Using that $|e^{\lambda s}-1|\le 3$ we have 
\[\int_{\kappa-i\infty}^{\kappa-iY}\int_{\kappa+iY}^{\kappa+i\infty}\mathbb{E}(L(1,\mathbb{X})^{s})(e^{\gamma}\tau)^{-s}\left(\frac{e^{\lambda s}-1}{\lambda s}\right)^N\frac{ds}{s}\ll \frac1N\left(\frac3{\lambda Y}\right)^N\mathbb{E}(L(1,\mathbb{X})^{\kappa})(e^{\gamma}\tau)^{-\kappa},\]
and similarly, together with Lemma \ref{relatechitoX} we obtain 
\begin{align*}
\int_{\kappa-i\infty}^{\kappa-iY}\int_{\kappa+iY}^{\kappa+i\infty}M(s)(e^{\gamma}\tau)^{-s}\left(\frac{e^{\lambda s}-1}{\lambda s}\right)^N\frac{ds}{s}&\ll \frac1N\left(\frac3{\lambda Y}\right)^NM(\kappa)(e^{\gamma}\tau)^{-\kappa} \\
&\ll \frac1N\left(\frac3{\lambda Y}\right)^N\mathbb{E}(L(1,\mathbb{X})^{\kappa})(e^{\gamma}\tau)^{-\kappa}.
\end{align*}
For the remaining parts of the integral we have that $|t|\le Y$ so we apply Lemma \ref{relatechitoX} which states that $M(s)-\mathbb{E}(L(1,\mathbb{X}))^s \ll \mathbb{E}(L(1,\mathbb{X})^{\Re s})/(\log|D|)^{11}$. Then use the inequality $|(e^{\lambda s}-1)/\lambda s|\le 4$ to obtain 
\[J_M(\tau)-J(\tau)\ll \frac1N\left(\frac3{\lambda Y}\right)^N\mathbb{E}(L(1,\mathbb{X})^{\kappa})(e^{\gamma}\tau)^{-\kappa}+\frac{Y}{\kappa}4^N\frac{\mathbb{E}(L(1,\mathbb{X})^{\kappa})}{(e^{\gamma}\tau)^{\kappa}(\log|D|)^{11}}.\]
Choosing $N=[\log_2|D|]$ and $\lambda=e^{10}/Y$ then \eqref{PhiXasymp} gives us that 
\begin{equation}\label{bounddiffJJM}
J_M(\tau)-J(\tau)\ll\frac{\Phi_{\mathbb{X}}(\tau)}{(\log |D|)^8}.
\end{equation}
On the other hand, by Theorem \ref{distfuncX} in combination with our choice for $\lambda$, $N$ and $Y$ we have 
\[\Phi_{\mathbb{X}}(e^{\pm\lambda N}\tau)=\Phi_{\mathbb{X}}(\tau)\left(1+O\left(\frac{e^{\tau}(\log_2|D|)^2\log_3|D|}{\log |D|}\right)\right).\]
Hence, combining \eqref{boundJ}, \eqref{boundJM} and \eqref{bounddiffJJM} 
\begin{align*}
\mathbb{P}(L(1,\chi_D)>e^{\gamma}\tau)&\le J_M(\tau)\\
&\le J(\tau)+O\left(\frac{\Phi_{\mathbb{X}}(\tau)}{(\log |D|)^8}\right)\\
&\le \Phi_{\mathbb{X}}(\tau)\left(1+O\left(\frac{e^{\tau}\log_2|D|\log_3|D|}{\log |D|}\right)\right),
\end{align*}
and 
\begin{align*}
\mathbb{P}(L(1,\chi_D)>e^{\gamma}\tau)&\ge J_M(e^{\lambda N}\tau)\\
&\ge J(e^{\lambda N}\tau)+O\left(\frac{\Phi_{\mathbb{X}}(\tau)}{(\log |D|)^8}\right)\\
&\ge \Phi_{\mathbb{X}}(\tau)\left(1+O\left(\frac{e^{\tau}(\log_2|D|)^2\log_3|D|}{\log |D|}\right)\right).
\end{align*}
The final step is done by recalling for $D\in\hn$ we have $|D|=q^{n}$.
\end{proof}
And now how to make use of Theorem \ref{distfuncchi} to prove the corollaries of Section \ref{applications:intro}.
\begin{proof}[Proof of Corollary \ref{disthD}]
We note by Artin's class number formula given by \eqref{Artinform} that $h_D\ge e^{\gamma}\tau\frac{\sqrt{|D|}}{\sqrt{q}}$ if and only if for $D\in\h$ we have $L(1,\chi_D)\ge e^{\gamma}\tau$. Specializing to $n=2g+1$ we see Theorem \ref{distfuncchi} proved that the number of $D$ such that $L(1,\chi_D)>e^{\gamma}\tau$ is given by
\[|\h|\Phi_{\mathbb{X}}(\tau)\left(1+O\left(\frac{e^{\tau}(\log_2|D|)^2\log_3|D|}{\log |D|}\right)\right).\]
Finally, we use Theorem \ref{distfuncX} to conclude that the number of $D$ such that $h_D\ge e^{\gamma}\tau\frac{\sqrt{|D|}}{\sqrt{q}}$ is given by 
\begin{align*}
&|\h|\exp\left(-C_1(q^{\{\log\kappa\}})\frac{q^{\tau-C_0(q^{\{\log\kappa\}})}}{\tau}\left(1+O\left(\frac{\log\tau}{\tau}\right)\right)\right)
\times \left(1+O\left(\frac{e^{\tau}(\log_2|D|)^2\log_3|D|}{\log |D|}\right)\right)\\
&=|\h|\exp\left(-C_1(q^{\{\log\kappa\}})\frac{q^{\tau-C_0(q^{\{\log\kappa\}})}}{\tau}\left(1+O\left(\frac{\log\tau}{\tau}\right)\right)\right),
\end{align*}
where the final estimate follows from the range of $\tau$. The analogous estimate for small values of $h_D$ follows along the same lines. 
\end{proof}
\section{Optimal $\Omega$-results: Proof of Theorem \ref{OmegaResults}}\label{Optomegsection}

For each irreducible polynomial $P\in \F$, let  $\delta_P\in \{-1, 1\}.$ Define $\mathcal{S}_N(n, \{\delta_P\})$ to be the set of all monic irreducibles $Q\in \F$ such that $\deg Q=N$ and 
$$ \left(\frac{P}{Q}\right)=\delta_P,$$
for all irreducibles $P$ with $\deg P\leq n$. We also let $\mathcal{P}(n)$ denote the product of all irreducible polynomials $P$ with $\deg P\leq n$.

\begin{lem}\label{SetPrimes}
Let $N$ be large, and  $1\leq n\leq (\log_q(N))^2$ be a real number. Then, we have
$$|\mathcal{S}_N(n, \{\delta_P\})|=\frac{q^N}{2^{\Pi_q(n)}N}+ O\left(q^{\frac{N}{2}+n}\right).$$
\end{lem}
\begin{proof}
For each monic polynomial $f\in \F$, define $\delta_f=\prod_{P\mid f} \delta_P$. Let $Q$ be an irreducible polynomial of degree $N$. Then, observe that
\begin{equation}\label{KeyOmega}
\sum_{f \mid \mathcal{P}(n)} \delta_f \left(\frac{f}{Q}\right)=\prod_{ \deg P\leq n} \left(1+\delta_P\left(\frac{P}{Q}\right)\right)=\begin{cases} 2^{\Pi_q(n)} &\text{ if } Q \in \mathcal{S}_N(n, \{\delta_P\}),\\ 0 &\text{ otherwise}.
\end{cases}
\end{equation}
Therefore, we deduce that 
$$|\mathcal{S}_N(n, \{\delta_P\})| =\frac{1}{2^{\Pi_q(n)}} \sum_{f \mid \mathcal{P}(n)} \delta_f\sum_{\substack{Q \text{ irreducible }\\ \deg Q=N}}  \left(\frac{f}{Q}\right).$$
Since all the divisors of $\mathcal{P}(n)$ are square-free, we obtain from  \eqref{CharSumGRH} that for all $f\neq 1$ such that $f\mid \mathcal{P}(n)$, we have
$$ \sum_{\substack{Q \text{ irreducible }\\ \deg Q=N}}  \left(\frac{f}{Q}\right) \ll \deg (f) q^{\frac{N}{2}}\ll q^{\frac{N}{2}+n}.$$
since 
\begin{equation}\label{DegF}
\deg f\leq \deg \mathcal{P}(n)=\sum_{j=1}^nj\pi_q(j)\asymp q^n,
\end{equation}
by the prime number theorem. Finally, since the number of divisors of $\mathcal{P}(n)$ is $2^{\Pi_q(n)}$ we deduce that
$$|\mathcal{S}_N(n, \{\delta_P\})|= \frac{\pi_q(N)}{2^{\Pi_q(n)}}+ O\left(q^{\frac{N}{2}+n} \right) $$
which completes the proof.
\end{proof}

We shall deduce Theorem \ref{OmegaResults} from the following proposition
\begin{pro}\label{Omega}
 We have
\begin{equation}\label{Optimal1} \sum_{Q\in \mathcal{S}_N(n, \{\delta_P\})} L(1, \chi_Q)= \zeta_{\mathbb{A}}(2)\frac{\pi_q(N)}{2^{\Pi_q(n)}}\prod_{\deg P\leq n} \left(1+\frac{\delta_P}{|P|}\right)+ O\left(N^2 q^{N/2+2n}\right).
\end{equation}
\end{pro}
\begin{proof}

 First, it follows from \eqref{CharSum} that for all $m\geq N$ we have
$$
L(1,\chi_Q)= \sum_{\deg F\leq m}\frac{\chi_Q(F)}{|F|}.
$$
Let $A=2N \deg \mathcal{P}(n) \ll N q^{n}$ by \eqref{DegF}. Then, from \eqref{KeyOmega} we obtain
\begin{equation}\label{SumL1}
\begin{aligned} 
\sum_{Q\in \mathcal{S}_N(n, \{\delta_P\})}L(1,\chi_Q)=\frac{1}{2^{\Pi_q(n)}}\sum_{f \mid \mathcal{P}(n)} \delta_f \sum_{\substack{Q \text{ irreducible }\\ \deg Q=N}}\left(\frac{f}{Q}\right)\sum_{\deg F\leq A}\frac{\left(\frac{Q}{F}\right)}{|F|}\\
=\frac{1}{2^{\Pi_q(n)}}\sum_{f \mid \mathcal{P}(n)} \delta_f \sum_{\deg F\leq A}\frac{1}{|F|}\sum_{\substack{Q \text{ irreducible }\\ \deg Q=N}}\left(\frac{Ff}{Q}\right),
\end{aligned}
\end{equation}
by quadratic reciprocity \eqref{QuadRecip}.
Since any divisor $f$ of $\mathcal{P}(n)$ is square-free, it follows that $Ff$ is a square only when $F=f h^2$, for some monic polynomial $h$. In this case, we have
$$ \sum_{\substack{Q \text{ irreducible }\\ \deg Q=N}}\left(\frac{Ff}{Q}\right)= \pi_q(N)+ O(\omega(F))= \pi_q(N)+O(A),
$$
where $\omega(F)$ is the number of irreducible divisors of $F$, and $\omega(F)\leq \deg F\leq A$. 

 Furthermore, if $Ff$ is not a square, then by \eqref{CharSumGRH} we get
$$ \sum_{\substack{Q \text{ irreducible }\\ \deg Q=N}}  \left(\frac{Ff}{Q}\right) \ll \deg (Ff) q^{\frac{N}{2}}\ll A q^{N/2},$$
by \eqref{DegF}.
Inserting these estimates in \eqref{SumL1}, we deduce 
\begin{equation}\label{L1Average2}
\sum_{Q\in \mathcal{S}_N(n, \{\delta_P\})}L(1,\chi_Q)=\frac{\pi_q(N)}{2^{\Pi_q(n)}}\sum_{f \mid \mathcal{P}(n)} \frac{\delta_f}{|f|} \sum_{\deg h\leq (A-\deg f)/2}\frac{1}{|h|^2}+ O\left(A^2q^{N/2}\right),
\end{equation}
since $$\sum_{\deg F\leq A}\frac{1}{|F|}= \sum_{k=1}^A \sum_{\deg F=k}\frac{1}{q^k}=A.$$
Finally, since $\deg f\leq \deg \mathcal{P}(n)\leq A/2$, then for all $f \mid \mathcal{P}(n)$ we have the tail of the inner sum is very small:
$$ 
\sum_{\deg h> (A-\deg f)/2}\frac{1}{|h|^2}\leq \sum_{\deg h> A/2} \frac{1}{|h|^2}\leq \sum_{k>A/2} \frac{1}{q^{2k}}\ll q^{-N}.
$$
Inserting this estimate in \eqref{L1Average2} completes the proof.
\end{proof}
We finish this section by proving Theorem \ref{OmegaResults}.
\begin{proof}[Proof of Theorem \ref{OmegaResults}]
We choose $n$ such that 
\begin{equation}\label{ChoiceNn}
\frac{N\log N}{10\zeta_{\mathbb{A}}(2) q}\leq q^n< \frac{N\log N}{10\zeta_{\mathbb{A}}(2)}.
\end{equation}
%
 We choose $\delta_P=1$ for all monic irreducibles $P$ with $\deg P\leq n$. Then, it follows from Lemma \ref{SetPrimes} and Proposition \ref{Omega} that 
\begin{equation}\label{AverageL1PrimeBig}\frac{1}{|\mathcal{S}_N(n, \{\delta_P\})|}\sum_{Q\in \mathcal{S}_N(n, \{\delta_P\})}  L(1, \chi_Q)= \zeta_{\mathbb{A}}(2)\prod_{\deg P\leq n} \left(1+\frac{1}{|P|}\right) \big(1+O\left(q^{-N/6}\right)\big).
\end{equation}
Furthermore, by 
Lemma \ref{Mertens} we have
$$ \zeta_{\mathbb{A}}(2)\prod_{\deg P\leq n} \left(1+\frac{1}{|P|}\right)= \prod_{\deg P\leq n} \left(1-\frac{1}{|P|}\right)^{-1}\left(1+O\left(\frac{q^{-n}}{n}\right)\right)=e^{\gamma} n +O(1).$$
Combining this estimate with \eqref{ChoiceNn} and \eqref{AverageL1PrimeBig} yield the existence of a monic irreducible $Q$ of degree $N$, such that
$$ L(1, \chi_Q)\geq  e^{\gamma} \log(N\log N)+O(1)= e^{\gamma}\left(\log_2 |Q|+ \log_3 |Q|\right)+O(1), $$
as desired. Finally, one can deduce \eqref{Negative1} along the same lines by taking $\delta_P=-1$ for all monic irreducibles $P$ with $\deg P\leq n$.
\end{proof}

\end{document}